\documentclass[12pt]{amsart}
\usepackage{amscd,amsmath,amsthm,amssymb}
\usepackage{amsfonts,amssymb,amscd,amsmath,enumerate,verbatim}
\usepackage{tikz, float} \usetikzlibrary {positioning}
\usepackage[left]{lineno}
\usepackage{pstcol,pst-plot,pst-3d}%\usepackage[T1]{fontenc}
\usepackage{color}
\usepackage{pstricks}
\usepackage{stmaryrd}
\usepackage[utf8]{inputenc}
\usepackage{pstricks-add}
\usepackage{graphicx}
\newpsstyle{fatline}{linewidth=1.5pt}
\newpsstyle{fyp}{fillstyle=solid,fillcolor=verylight}
\definecolor{verylight}{gray}{0.97}
\definecolor{light}{gray}{0.9}
\definecolor{medium}{gray}{0.85}
\definecolor{dark}{gray}{0.6}

 \def\NZQ{\mathbb}               % the font for N,Z,Q,R,C
 \def\NN{{\NZQ N}}
 
 \def\ZZ{{\NZQ Z}}

 \def\FF{{\NZQ F}}
 \def\GG{{\NZQ G}}
 \def\HH{{\NZQ H}}

 \def\frk{\mathfrak}               % font for "Fraktur"

 \def\mm{{\frk m}}

 \def\Rc{{\mathcal R}}
 
 \def\G{{\mathcal G}}
 \def\F{{\mathcal F}}

  \def\Gc{{\mathcal G}}
  \def\Ac{{\mathcal A}}
 \def\R{{\mathcal R}}
\def\Bc{{\mathcal B}}
\def\Cc{{\mathcal C}}
\def\Wc{{\mathcal W}}
\def\Sc{{\mathcal S}}
 \def\ab{{\mathbf a}}
 \def\bb{{\mathbf b}}
 \def\xb{{\mathbf x}}
 \def\yb{{\mathbf y}}
 \def\zb{{\mathbf z}}
 \def\cb{{\mathbf c}}
 \def\db{{\mathbf d}}
 
 \def\wb{{\mathbf w}}

 \def\opn#1#2{\def#1{\operatorname{#2}}} % to make operators
 \opn\chara{char} \opn\length{\ell} \opn\pd{pd} \opn\rk{rk}
 \opn\projdim{proj\,dim} \opn\injdim{inj\,dim} \opn\rank{rank}
 \opn\depth{depth} \opn\grade{grade} \opn\height{height}
 \opn\embdim{emb\,dim} \opn\codim{codim}
 
 \opn\Tr{Tr} \opn\bigrank{big\,rank}
 \opn\superheight{superheight}\opn\lcm{lcm}
 \opn\trdeg{tr\,deg}%\emph{
 \opn\reg{reg} \opn\lreg{lreg} \opn\ini{in} \opn\lpd{lpd}
 \opn\size{size} \opn\sdepth{sdepth}
 \opn\link{link}\opn\fdepth{fdepth}\opn\lex{lex}
 \opn\tr{tr}
 \opn\type{type}
 \opn\gap{gap}
 \opn\arithdeg{arith-deg}
 \opn\revlex{revlex}
 \opn\div{div} \opn\Div{Div} \opn\cl{cl} \opn\Cl{Cl}
 \opn\Spec{Spec} \opn\Supp{Supp} \opn\supp{supp} \opn\Sing{Sing}
 \opn\Ass{Ass} \opn\Min{Min}\opn\Mon{Mon}
 \opn\Ann{Ann} \opn\Rad{Rad} \opn\Soc{Soc}
 \opn\Im{Im} \opn\Ker{Ker} \opn\Coker{Coker} \opn\Am{Am}
 \opn\Hom{Hom} \opn\Tor{Tor} \opn\Ext{Ext} \opn\End{End}
 \opn\Aut{Aut} \opn\id{id}
 \def\F{{\mathcal F}}
 \opn\nat{nat}
 \opn\pff{pf}%   \pf exists already
 \opn\Pf{Pf} \opn\GL{GL} \opn\SL{SL} \opn\mod{mod} \opn\ord{ord}
 \opn\Gin{Gin} \opn\Hilb{Hilb}\opn\sort{sort}
 \opn\PF{PF}\opn\Ap{Ap}
 \opn\mult{mult}
 \opn\bight{bight}
 \opn\aff{aff}
 \opn\relint{relint} \opn\st{st}
 \opn\lk{lk} \opn\cn{cn} \opn\core{core} \opn\vol{vol}  \opn\inp{inp} \opn\nilpot{nilpot}
 \opn\link{link} \opn\star{star}\opn\lex{lex}\opn\set{set}
 \opn\width{wd}
 \opn\Fr{F}
 \opn\QF{QF}
 \opn\G{G}
 \opn\type{type}\opn\res{res}
 \opn\conv{conv}
 \opn\Deg{Deg}
 \opn\Sym{Sym}
 \opn\gr{gr}
 
 \def\pot#1#2{#1[\kern-0.28ex[#2]\kern-0.28ex]}

 \opn\dirlim{\underrightarrow{\lim}}
 \opn\inivlim{\underleftarrow{\lim}}
 \let\iso=\cong
 \let\Union=\bigcup
 
 \let\Dirsum=\bigoplus
 
 \let\to=\rightarrow
 \let\To=\longrightarrow
 \def\Implies{\ifmmode\Longrightarrow \else
         \unskip${}\Longrightarrow{}$\ignorespaces\fi}
 \def\implies{\ifmmode\Rightarrow \else
         \unskip${}\Rightarrow{}$\ignorespaces\fi}
 \def\iff{\ifmmode\Longleftrightarrow \else
         \unskip${}\Longleftrightarrow{}$\ignorespaces\fi}

 \let\:=\colon
 \newtheorem{Theorem}{Theorem}[section]
 \newtheorem{Lemma}[Theorem]{Lemma}
 \newtheorem{Corollary}[Theorem]{Corollary}
 \newtheorem{Proposition}[Theorem]{Proposition}
 \newtheorem{Remark}[Theorem]{Remark}
 
 \newtheorem{Example}[Theorem]{Example}

 \let\epsilon\varepsilon
 \let\kappa=\varkappa
 \def\qed{\ifhmode\textqed\fi
       \ifmmode\ifinner\quad\qedsymbol\else\dispqed\fi\fi}
 \def\textqed{\unskip\nobreak\penalty50
        \hskip2em\hbox{}\nobreak\hfil\qedsymbol
        \parfillskip=0pt \finalhyphendemerits=0}
 \def\dispqed{\rlap{\qquad\qedsymbol}}
 
 \opn\dis{dis}
 \def\pnt{{\raise0.5mm\hbox{\large\bf.}}}
 
 \opn\Lex{Lex}

\begin{document}

\title{Componentwise linear powers and the $x$-condition}

\author {J\"urgen Herzog, Takayuki Hibi  and  Somayeh Moradi }

\address{J\"urgen Herzog, Fachbereich Mathematik, Universit\"at Duisburg-Essen, Campus Essen, 45117
Essen, Germany} \email{juergen.herzog@uni-essen.de}

\address{Takayuki Hibi, Department of Pure and Applied Mathematics,
Graduate School of Information Science and Technology, Osaka
University, Suita, Osaka 565-0871, Japan}
\email{hibi@math.sci.osaka-u.ac.jp}

\address{Somayeh Moradi, Department of Mathematics, School of Science, Ilam University,
P.O.Box 69315-516, Ilam, Iran}
\email{so.moradi@ilam.ac.ir}

\dedicatory{ }
 \keywords{componentwise linear, $x$-condition, symmetric powers, vertex cover ideals}
     \subjclass{Primary 05E40, 13A02;  Secondary 13P10}
\thanks{The second author was supported by JSPS KAKENHI 19H00637.}

\begin{abstract}
Let $S=K[x_1,\ldots,x_n]$ be the polynomial ring over a field and   $A$ a standard graded  $S$-algebra. In terms of the Gr\"obner basis of the defining ideal $J$ of $A$ we give a condition, called the x-condition, which implies that all graded  components $A_k$  of $A$ have linear quotients and with additional assumptions are componentwise linear. A typical example of such an algebra is the Rees ring $\R(I)$ of a graded ideal or the symmetric algebra $\Sym(M)$ of a module $M$. We apply our criterion to study certain symmetric algebras and the powers  of vertex cover ideals of certain classes of graphs.
\end{abstract}

\maketitle

\setcounter{tocdepth}{1}
%\tableofcontents

\section*{Introduction}

Let $K$  be a field,   $S = K[x_1,\ldots, x_n]$ be the polynomial ring in $n$ variables over $K$, and
let $\mm =(x_1,\ldots, x_n)$ be its graded maximal ideal. In 1999\, the first and second author of this paper introduced in \cite{HH2} the concept of componentwise linear ideals. Assuming that $K$ is of characteristic $0$, componentwise linear ideals  $I\subset S$ are characterized by the remarkable property that they are graded and $\beta_{i,j}(I)=\beta_{i,j}(\Gin(I))$, see \cite[Theorem 1.1]{AHH1}. Here, $\beta_{i,j}(M)$ denote the graded Betti numbers of a graded $S$-module $M$ and $\Gin(I)$ the generic initial ideal of $I$ with respect to the reverse lexicographic order. A componentwise linear ideal $I\subset S$  which does not contain a variable has another  interesting property, namely that it is Golod (see \cite[Theorem 4]{HRW}), so that in particular the Poincare series of $S/I$ is a rational function.

A finitely generated graded $S$-module $M$ is called {\em componentwise linear}, if for each integer  $j$ the submodule of $M$ generated by all elements of degree $j$ in $M$ has linear resolution.
It can be easily seen that modules with linear resolution are componentwise linear. New and interesting situations occur when not all generators of the module are of same degree. Nevertheless, in this more general case, componentwise linear ideals behave with respect to several properties very much like modules with linear resolution. A first such example is the following: we say that  $M$ has linear quotients,  if there exists  a system of homogeneous generators $f_1,\ldots,f_m$  of $M$ with the property that each of the colon ideals $(f_1,\ldots,f_{j-1}):f_j$ is generated by linear forms. It is well-known that if $M$ is generated in a single degree $d$ and $M$ has linear quotients with respect to $f_1,\ldots,f_m$  where each  $f_i$ is of degree $d$, then $M$ has linear resolution. For  a graded $S$-module,  which is not necessarily generated in a single degree,  we show in Proposition~\ref{componentwise} that it is componentwise linear  if it has linear quotients with respect to a sequence  $f_1,\ldots,f_m$  satisfying the additional condition  that $\deg f_1\leq \deg f_2 \leq \cdots \leq \deg f_m$.
The ideal $I=(x_1^2,x_2^2)$ is not componentwise linear,  but it admits  the following sequence of generators with linear quotients $x_1^2,x_1x_2^2,x_2^2$. This shows that in Proposition~\ref{componentwise} we need the degree assumptions on the $f_i$. The sequence  in this example is also bad, as  it is not a minimal system of generators of $I$. Surprisingly,  a graded ideal is componentwise linear if it  is minimally generated by a sequence with  linear quotients, see \cite[Theorem 8.2.15]{HH}. This statement as well as its proof is word by word the same for graded modules. Summarizing we can say the following: let $M$ be a graded $S$-module and $f_1,\ldots,f_m$ a system of generators  with  linear quotients for $M$. Then $M$ is componentwise linear if either $f_1,\ldots,f_m$ is a minimal system of generators of $M$ or $\deg f_1\leq \deg f_2\leq \cdots \leq \deg f_m$.
In Section~1 we put together what we need regarding linear quotients, and show in addition to the above  mentioned  Proposition~\ref{componentwise} that if the linear quotient sequence $f_1,\ldots,f_m$  of the module $M$ is a minimal system of generators of $M$, then the projective dimension and the regularity of $M$ can be expressed in terms of the $f_i$ and their linear quotients, see Proposition~\ref{generalbetti}. The proof is very similar to the corresponding result for ideals, due to Sharifan and Varbaro \cite[Corollary 2.4]{SV}. It is also observed that $M$ has linear quotients if for a suitable monomial order, the initial module of  the first syzygy module of $M$
is at most linear, see ~Proposition~\ref{linear quotients} for the precise statement.

The second instance where componentwise linear modules behave like  modules with linear resolution is the following: let $\FF: 0\to F_p\to \cdots \to F_1\to F_0\to 0$ be the graded minimal free $S$-resolution of $M$. One defines a filtration on $\FF$ by setting $\F_j F_i=\mm^{j-i} F_i$ for all $i$ and $j$. It is obvious that  the associated graded complex $\gr_F(\FF)$ is acyclic if $M$ has linear resolution. In \cite{HI}  a module $M$ is called a  {\em Koszul module}, if $\gr_F(\FF)$ is acyclic. Thus,  modules with linear resolution are Koszul. R\"omer \cite[Theorem 3.2.8]{R1} and Yanagawa
\cite[Proposition 4.9]{Y}
 independently proved the beautiful result which says that $M$ is componentwise linear if and only if $M$ is Koszul.

In this paper we focus on powers of componentwise linear ideals and symmetric powers of componentwise linear modules, and add another instance of similar behaviour of linearity  and componentwise linearity  by   showing that the so-called $x$-condition for equigenerated ideals admits a  natural extension which allows to study powers of componentwise linear modules. For this purpose we consider in Section~2 a bigraded $K$-algebra $A$. Let $K$ be a field and $A=\Dirsum_{i,j}A_{(i,j)}$ be a bigraded $K$-algebra with $A_{(0,0)}=K$. We set  $A_j=\Dirsum_iA_{(i,j)}$. Then $A=\Dirsum_jA_j$ has the structure of a graded $A_0$-algebra, where $A_0$ is a graded $K$-algebra and each $A_j$ is a graded $A_0$-module with grading  $(A_j)_i=A_{(i,j)}$ for all $i$.

It is required that
\begin{enumerate}
\item[(i)] $A_0$ is the polynomial ring $S= K[x_1,\ldots,x_n]$ with the standard grading;

\item[(ii)] $A_1$  is finitely generated and $(A_1)_i=0$ for $i<0$;

\item[(iii)] $A=S[A_1]$, i.e., $A_i=A_1^i$ for all $i\geq 1$.
\end{enumerate}

A typical example of such an algebra is the Rees ring of a graded ideal $I\subset S$ or the symmetric algebra of a finitely generated graded $S$-module.

We fix a system of homogeneous generators $f_1,\ldots,f_m$ of $A_1$ with $\deg f_i=d_i$ for $i=1,\ldots,m$.
Let $T=K[x_1,\ldots,x_n,y_1,\ldots,y_m]$  be the bigraded polynomial ring   over $K$ in the indeterminates $x_1,\ldots, x_n, y_1,\ldots, y_m$ with
$\deg x_i=(1,0)$ for $i=1,\ldots,n$ and $\deg y_j=(d_j,1)$  for  $j=1,\ldots,m$.

We define the $K$-algebra homomorphism $\varphi\: T\to A$ with $\varphi(x_i)=x_i$ for $i=1,\ldots,n$ and $\varphi(y_j)=f_j$ for $j=1,\ldots, m$.  Then $\varphi$ is a surjective $K$-algebra homomorphism of bigraded $K$-algebras, and hence  $J=\Ker(\varphi)$ is a bigraded ideal in $T$. We say that the defining ideal $J$ of $A$ satisfies the {\em  $x$-condition}, with respect to a monomial order $<$ on $T$  if all $u\in G(\ini_<(J))$ are of the form $vw$ with $v\in S$ of degree $\leq 1$ and
$w\in K[y_1,\ldots,y_m]$.  As a main general result of this paper we show in Theorem~\ref{main} that if $J$ satisfies the $x$-condition for a specified monomial order on $T$, then for all $k\geq 1$, the graded $S$-modules $A_k$ have linear quotients.  This theorem generalizes a similar result \cite[Theorem 10.1.9]{HH} given for monomial ideals generated in a single degree. In the course of the proof of Theorem~\ref{main},  for each $k$ an explicit linear quotient sequence for $A_k$ is constructed.  In order to conclude then  that $A_k$ is componentwise linear it is required that this sequence either satisfies (i) the non-decreasing  degree condition,  or (ii) is a minimal set of generators for $A_k$. Condition (i) is satisfied if the  monomial order defined on $T$ is a weighted monomial order, see Corollary~\ref{weightedworks} for the precise statement. In particular (i)  is satisfied if all $f_i$ are of same degree. Condition (ii) is satisfied   if the Gr\"obner basis of $J$ is contained in $(x_1,\ldots,x_n)T$ (see Proposition~\ref{carisokay}),  and more important if  the initial ideal of $J$ is generated by monomials of degree $2$,  as shown in Theorem~\ref{monomialcase}. This condition on the initial ideal  is fulfilled  for the vertex cover ideal of the graphs which we study in Section~4.

In Section 3 we apply the theory developed is Section 2 to symmetric powers of modules and consider the interesting case of a module $M$ which is generated by any subset $\Sc$ of the canonical generators of the first Koszul cycles of the maximal ideal of $S$. Such a subset $\Sc$ is in bijection to the edges of a graph $G$ on the vertex set $\{1,\ldots,n\}$, and it turns out that the defining ideal $J$ of the symmetric algebra of the corresponding  module is just the binomial edge ideal of $G$. Since the initial ideal of binomial edge ideals are known
(\cite[Theorem 2.1]{HHHKR} and
\cite[Theorem 3.2]{Oht}), the $x$-condition  can be expressed in combinatorial terms of the underlying graph. Denoting  by $M_G$, the module $M$ which is defined by $G$ it is shown in Theorem~\ref{mg} that $\Sym_k(M_G)$ has linear resolution if and only if $G$ is chordal.
Theorem~\ref{additional} gives some additional information about the depth of the modules $\Sym_k(M_G)$.  We conclude this section by  Proposition~\ref{observation} which says that if $M$ is a graded module and the defining ideal $J$ of $\Sym(M)$ satisfies the $x$-condition, then all symmetric powers of $M$ are componentwise linear.

In Section 4 we consider vertex cover ideals of certain classes of graphs and show that the defining ideal $J$ of their Rees algebras satisfies the $x$-condition. The families of graphs which we consider are biclique  graphs, path graphs and Cameron--Walker graphs whose bipartite graph is complete. In each of these cases we compute the Gr\"{o}bner bases of $J$ and show that $\ini_<(J)$ is generated by quadratic monomials. By Theorem~\ref{monomialcase} this implies that all the powers of these vertex cover ideals are componentwise linear. To determine the Gr\"{o}bner bases of $J$ in these cases is quite involved and the answer is pretty complex. This is one reason why for example for chordal graphs we could not yet go beyond path graphs. But we expect that the powers of the vertex cover ideal of any chordal graph are all componentwise linear. Other cases of graphs whose powers of its  vertex cover ideals are componentwise linear have been considered in \cite{EQ}, \cite{HHO1}, \cite{M1} and \cite{M}.  In paper  \cite{HHO1} it is taken advantage of the property that componentwise linear ideals are Koszul.

\section{Componentwise linear modules}
\label{one}

\bigskip

Let  $S = K[x_1,\ldots, x_n]$ the polynomial ring over the field $K$, $M$ be a finitely generated $S$-module and $f_1,\ldots,f_m$ be a system of generators for $M$. For any $1\leq j\leq m$, set $I_j=(f_1,\ldots,f_{j-1}):f_j$, where $(f_1,\ldots,f_{j-1})$ is the submodule of $M$ generated by $f_1,\ldots,f_{j-1}$.  We say that $M$ has \emph{linear quotients} with respect to $f_1,\ldots,f_m$, if the
ideal $I_j$ is either $S$ or is generated by linear forms for any $1\leq j\leq m$. Moreover, $M$ is said to have linear quotients, when it has linear quotients with respect to some system of its generators.
In the following, for a graded $S$-module $M$, we denote by $M_{\langle  k\rangle }$, the submodule of $M$ generated by all elements $m\in M$ of degree $k$.

\begin{Proposition}\label{componentwise}
Let $M$ be a finitely generated graded $S$-module admitting  linear quotients with respect to $f_1,\ldots,f_m$, and assume that   $\deg f_1 \leq \deg f_2 \leq \cdots\leq \deg f_m $. Then
$M$ is componentwise linear.
\end{Proposition}

\begin{proof}
Set $M_0=(0)$, $M_j=(f_1,\ldots,f_j)$ for $1\leq j\leq m$, $I_j=M_{j-1}:f_j$ and $d_j=\deg f_j$. We prove the theorem for $M_j$ by induction on $j$. Let $j=1$. Then $M_1=S/I_1$.
By assumption $I_1$ is generated by linear forms and then $M_1$ has linear resolution. Consider the exact sequence of graded $S$-modules
$$0 \rightarrow M_{j-1}\rightarrow M_j\rightarrow (S/I_j)(-d_j)\rightarrow 0.$$
By induction we may assume that $M_{j-1}$ is componentwise linear. Let $k<d_j$. Then $(M_j)_{\langle k\rangle}=(M_{j-1})_{\langle k\rangle}$. Therefore by induction
$(M_j)_{\langle k\rangle}$ has linear resolution for $k<d_j$. The above exact sequence gives rise to the exact sequence

\begin{equation}\label{1}
0\rightarrow (M_{j-1})_{\langle d_j\rangle}\rightarrow (M_j)_{\langle d_j\rangle}\rightarrow (S/I_j)(-d_j)\rightarrow 0.
\end{equation}

By induction hypothesis $(M_{j-1})_{\langle d_j\rangle}$ has linear resolution. Also $S/I_j$ has a linear resolution, as discussed above. So
the exact sequence (\ref{1}) shows that $(M_j)_{\langle d_j\rangle}$ has linear resolution.
Since $d_j$ is the highest degree of the generators of $M_{j}$, it follows that $(M_j)_{\langle k\rangle}=\mm^{k-d_j}(M_j)_{\langle d_j\rangle}$ for $k\geq d_j$. This shows that $M_j$ is componentwise linear.
\end{proof}

For ideals with linear quotients the graded Betti numbers have been  determined in terms of the linear quotients \cite[Corollary 2.7]{SV}. Following  the same line of arguments one can prove a similar formula for modules with linear quotients. For the convenience of the reader we give the complete proof here.

\begin{Proposition}
\label{generalbetti}
Let $M$ be a finitely generated graded $S$-module admitting  linear quotients with respect to $f_1,\ldots,f_m$, and assume that  $f_1,\ldots,f_m$ is a minimal system of generators of $M$ and that  $\deg f_1 \leq \deg f_2 \leq \cdots\leq \deg f_m $. Let  $M_0=(0)$  and $M_j=(f_1,\ldots,f_j)$ for $1\leq j\leq m$. Furthermore, let  $I_j=M_{j-1}:f_j$  and  $\mu_j=\mu(I_j)$ for $j=1,\ldots,m$. Then
\begin{enumerate}
\item[{\em (a)}] $\beta_{i,i+j}(M)=\sum_{k,\deg(f_k)=j}{\mu_k\choose i}$.
\item[{\em (b)}] $\projdim M=\max\{\mu_k\:\; k=1,\ldots,m\}$.
\item[{\em (c)}] $\reg(M)=\deg(f_m)$.
\end{enumerate}
\end{Proposition}

\begin{proof}
It suffices to prove (a). The statements (b) and (c) follow directly from (a). By  Proposition~\ref{componentwise},  $M$ is  componentwise linear. Therefore  it follows from \cite[Proposition 8.2.13]{HH}  that
\begin{eqnarray}
\label{formula}
\beta_{i,i+j}(M)=\beta_i(M_{\langle j\rangle})-\beta_i(\mm M_{\langle j-1\rangle}).
\end{eqnarray}
In  \cite[Proposition 8.2.13]{HH}  this statement is actually given for componentwise linear ideals. But the proof for componentwise linear modules is word by word the same.

Assume first that there is no $f_k$ with $\deg(f_k)=j$. Then $M_{\langle j\rangle}=\mm M_{\langle j-1\rangle}$, and by (\ref{formula}) we have $\beta_{i,i+j}(M)=0$, in  accordance with the formula in (a).

Next assume that there is some  $f_k$ with $\deg(f_k)=j$.
Let $a=\min\{k\:\; \deg(f_k)=j\}$ and $b=\max\{k\:\; \deg(f_k)=j\}$. We show that
\begin{eqnarray}
\label{recursive}
\beta_i(M_{\langle j\rangle})=\beta_i(\mm M_{\langle j-1\rangle})+ \sum_{k=a}^b{\mu_k\choose i}.
\end{eqnarray}
This together with (\ref{formula}) yields then the desired conclusion.

For any integer $c$ with $a\leq c\leq b$, if $M_c=(f_1,\ldots,f_c)$, then
\[
(M_c)_{\langle j\rangle} =\mm M_{\langle j-1 \rangle}+ (f_a,\ldots,f_c).
\]
Since $M_c$ is generated by the linear quotient sequence $f_1,\ldots,f_c$, each $M_c$ is componentwise linear by  Propostion~\ref{componentwise}. In particular,
$(M_{c-1})_{\langle j\rangle}$ and $(M_{c})_{\langle j\rangle}$ have $j$-linear resolution, and also $(S/I_c)(-j)$ has $j$-linear resolution, and we get a short  exact sequence of modules
\[
0\to (M_{c-1})_{\langle j\rangle}\to (M_{c})_{\langle j\rangle}\to (S/I_c)(-j)\to 0.
\]
Since $$\Tor_{i+1}(K,(S/I_c)(-j))_{i+1+(j-1)}=\Tor_{i-1}(K,(M_{c-1})_{\langle j\rangle})_{i-1 +(j+1)}=0,$$
the long exact  sequence arising from the above short exact sequence gives us the short exact sequence
\[
0\to \Tor_i(K,(M_{c-1})_{\langle j\rangle})_{i+j}\to \Tor_i(K,(M_{c})_{\langle j\rangle})_{i+j}\to \Tor_i(K,(S/I_c)(-j))_{i+j}\to 0.
\]
Since $\dim_K \Tor_i(K,(S/I_c)(-j))_{i+j} ={\mu_c\choose i}$, see that for all $c=a,\ldots, b$ we have
\[
\beta_{i,i+j}((M_{c})_{\langle j\rangle})=\beta_{i,i+j}((M_{c-1})_{\langle j\rangle})+{\mu_c \choose i}.
\]
This together with the equality $M_{\langle j\rangle}=(M_{b})_{\langle j\rangle}$ implies (\ref{recursive}), and completes  the proof.
\end{proof}

Fix any monomial order $<$ on $S$ and let $F$ be a free $S$-module with a
basis $e_1,\ldots,e_m$. We define a monomial order on $F$ by setting $ue_i < ve_j$
if $i < j$, or $i = j$ and $u < v$, where $u$ and $v$ are monomials of $S$.

\begin{Proposition}\label{linear quotients}
Let $M$ be a finitely generated graded $S$-module and  $f_1,\ldots,f_m$ be a minimal system of homogeneous generators of $M$. Let $F=\bigoplus_{i=1}^m Se_i$ be the free $S$-module with basis $e_1,\ldots,e_m$. Let $\varphi:F\rightarrow M$ be the homomorphism of $S$-modules with $\varphi(e_i)=f_i$ and $C=\Ker(\varphi)$. Fix an order $<$ on $F$ induced by $e_1<\cdots<e_m$ and set $I_j=(f_1,\ldots,f_{j-1}):f_j$ for $1\leq j\leq m$. Then
\begin{enumerate}
  \item[{\em (a)}] $\ini_{<}(C)=\bigoplus_{j=1}^m \ini_{<}(I_j)e_j$.
  \item [{\em (b)}] If the generators of $\ini_{<}(C)$ are of the form $ue_i$ with $\deg u \leq 1$, then $M$ has linear quotients.
\end{enumerate}
\end{Proposition}

\begin{proof}
(a) It is obvious that $\bigoplus_{j=1}^m \ini_{<}(I_j)e_j\subseteq \ini_{<}(C)$.
Conversely, let $ue_j\in \ini_{<}(C)$, where $u$ is a monomial. Then by the definition of the monomial order, there exist homogeneous polynomials $g,g_1,\ldots,g_{j-1}\in S$ such that $ge_j-\sum_{\ell=1}^{j-1} g_{\ell}e_{\ell}\in C$ and $\ini(g)=u$. Since $gf_j=\sum_{\ell=1}^{j-1} g_{\ell}f_{\ell}$, it follows that $g\in I_j$. Thus $ue_j\in \ini_{<}(I_j)e_j$.

(b) The hypothesis together with (a) implies that $\ini(I_j)$ is generated by variables. Therefore the reduced Gr\"{o}bner basis of $I_j$ is
generated by linear forms, because the leading terms of the reduced Gr\"{o}bner basis are the generators of $\ini(I_j)$ and the reduced Gr\"{o}bner basis consists of homogeneous polynomials. Since the reduced Gr\"{o}bner basis of $I_j$ contains the minimal set of generators of $I_j$, the ideal $I_j$ is generated by linear forms.
\end{proof}

\section{The $x$-condition}

Let $K$ be a field and $A=\Dirsum_{i,j}A_{(i,j)}$ be a bigraded $K$-algebra with $A_{(0,0)}=K$. Set $A_j=\Dirsum_iA_{(i,j)}$. Then $A=\Dirsum_jA_j$ has the structure of a graded $A_0$-algebra, where $A_0$ is a graded $K$-algebra and each $A_j$ is a graded $A_0$-module with grading  $(A_j)_i=A_{(i,j})$ for all $i$.

We now require in addition that
\begin{enumerate}
\item[(i)] $A_0$ is the polynomial ring $S= K[x_1,\ldots,x_n]$ with the standard grading;

\item[(ii)] $A_1$  is finitely generated and $(A_1)_i=0$ for $i<0$;

\item[(iii)] $A=S[A_1]$, i.e., $A_i=A_1^i$ for all $i\geq 1$.
\end{enumerate}

We fix a system of homogeneous generators $f_1,\ldots,f_m$ of $A_1$ with $\deg f_i=d_i$ for $i=1,\ldots,m$.
% and assume that $d_1\leq d_2\leq \cdots \leq d_m$.

Let $T=K[x_1,\ldots,x_n,y_1,\ldots,y_m]$  be the bigraded polynomial ring   over $K$ in the indeterminates $x_1,\ldots, x_n, y_1,\ldots, y_m$ with
$\deg x_i=(1,0)$ for $i=1,\ldots,n$ and $\deg y_j=(d_j,1)$  for  $j=1,\ldots,m$.

We define the $K$-algebra homomorphism $\varphi\: T\to A$ with $\varphi(x_i)=x_i$ for $i=1,\ldots,n$ and $\varphi(y_j)=f_j$ for $j=1,\ldots, m$.  Then $\varphi$ is a surjective $K$-algebra homomorphism of bigraded $K$-algebras, and hence  $J=\Ker(\varphi)$ is a bigraded ideal in $T$.

We say that the defining ideal $J$ of $A$ satisfies the {\em  $x$-condition}, with respect to a monomial order $<$ on $T$  if all $u\in G(\ini_<(J))$ are of the form $vw$ with $v\in S$ of degree $\leq 1$ and
$w\in K[y_1,\ldots,y_m]$. Here $G(I)$ denotes the minimal set of monomial generators of the monomial ideal $I$.

An important example of a bigraded  algebra $A$ as above is,  the Rees ring $\R(I)=\Dirsum_{k\geq 0}I^kt^k$ of a graded ideal. In that case, $A_1=It$, and if $I=(f_1,\ldots, f_m)$ with $\deg f_i=d_i$, then $A_1$ is generated by $f_1t,\ldots,f_mt$ with $\deg f_it=(d_i,1)$ for all $i$.

In \cite[Corollary 1.2]{HHZ} it is shown that if $I$ is generated in a single degree and if the defining ideal $J$ of $\R(I)$ satisfies the $x$-condition, then all powers of $I$ have linear resolution.

With exactly the same proof one obtains the following more general result.

\begin{Theorem}
\label{exactlythesame}
Suppose that all generators of $A_1$ have the same degree and that the defining ideal $J$ of $A$ satisfies the $x$-condition. Then  $A_k$ has linear resolution for all $k\geq 1$.
\end{Theorem}

Next we turn to case that not all generators of $A_1$ may have the same degree.

Let $<$  be a monomial order on $T$. A monomial $u$ in $T$
is called {\em standard with respect to} $<$, if $u\not\in  \ini_<(J)$.

\begin{Lemma}
\label{standard}
Fix a monomial order $<$ on $T$ and an integer $k\geq 1$,  and let $\{w_1, \ldots, w_r\}$ be the set of monomials of degree $k$ in the variables $y_1, \ldots,y_m$ in $T$. We may assume that $w_1<w_2 <\cdots <w_r$.
Let $w_i$ be a non-standard monomial. Then there exist standard monomials $w_{j_1},\ldots,w_{j_t}$   with $j_1<\cdots<j_t<i$ and homogeneous polynomials $p_{\ell}\in S$ such that $w_i-\sum_{\ell=1}^t p_{\ell} w_{j_{\ell}}\in J$.

Moreover,  the set $$\{\varphi(w_i):\ w_i \ \textrm{is  standard}\}$$ is a  system of homogeneous  generators of $A_k$.
\end{Lemma}

\begin{proof}
 Let  $w_{s_1},\ldots,w_{s_d}$ be all the non-standard monomials of degree $k$, where $s_1<\cdots<s_d$. We prove the first part of the lemma for the monomials $w_{s_a}$ by induction on $a$. Let $a=1$.   Then there is a polynomial $w_{s_1}-\sum_{\ell=1}^{s_1-1}  q_{\ell} w_{{\ell}}\in J$ with the initial term $w_{s_1}$, where $q_{\ell}$ are homogeneous polynomials in $S$. By our assumption on $s_1$, all $w_{{\ell}}$ are standard for $\ell<s_1$.

Now let $a>1$.   Since $w_{s_a}$ is non-standard,  there is a polynomial
$w_{s_a}-\sum_{\ell=1}^{s_a-1}  q'_{\ell} w_{{\ell}}\in J$ with the initial term $w_{s_a}$, where $q'_{\ell}$ are homogeneous polynomials in $S$.
By induction we may assume that any $w_{s_i}$ with $i<a$ in this sum  can be replaced modulo $J$ by $\sum_{\ell=1}^{s_i-1}  q_{i,\ell} w_{{\ell}}$, where $q_{i,\ell}=0$ when $w_{{\ell}}$ is non-standard. This proves the first part of the lemma.

Note that $\{\varphi(w_1),\ldots, \varphi(w_r)\}$ is a system of homogeneous  generators of $A_k$.
Suppose $w_i$ is a non-standard monomial. Then $w_i-\sum_{\ell=1}^t p_{\ell} w_{j_{\ell}}\in J$,  as stated in the lemma. This means that $\varphi(w_i)=\sum_{\ell=1}^t p_{\ell} \varphi(w_{j_{\ell}})$, which proves the second part of the lemma.
\end{proof}

Let $<'$ denote a monomial  order on $K[y_1,\ldots,y_m]$, $<_x$ be a monomial order on $K[x_1,\ldots,x_n]$  and  let $<$ be an order on $T$ such that
\begin{eqnarray}
\label{monomialorder}
\prod_{i=1}^n x_i^{a_i}\prod_{i=1}^m y_i^{b_i} &<& \prod_{i=1}^n x_i^{a'_i} \prod_{i=1}^m y_i^{b'_i},
\end{eqnarray}
if
\begin{eqnarray*}
\prod_{i=1}^m y_i^{b_i} <'\prod_{i=1}^m y_i^{b'_i}\quad \text{or}\quad  \prod_{i=1}^m y_i^{b_i}&=&\prod_{i=1}^m y_i^{b'_i}\quad  \text{and} \quad
 \prod_{i=1}^n x_i^{a_i} <_x \prod_{i=1}^n x_i^{a'_i}.
\end{eqnarray*}

The following theorem generalizes \cite[Theorem 10.1.9]{HH}.

\begin{Theorem}\label{main}
We maintain the  notation and the   assumptions introduced before, and suppose that  $J$ satisfies the $x$-condition with respect to the monomial order defined in  {\em (\ref{monomialorder})}.  Then for all $k\geq 1$, the $S$-module  $A_k$ has linear quotients.
\end{Theorem}

\begin{proof}
Let $k\geq 1$. By Lemma~\ref{standard},  $A_k$ has a system of generators $h_1,\ldots,h_s$, where each of them is of the form $f_{i_1}\cdots f_{i_k}$ and such that $h^*=y_{i_1}\cdots y_{i_k}$ is a standard monomial of $T$ with respect to $<$.  We may assume that $h^*_1<'\cdots<'h^*_s$. For $j=1,\ldots,s$, let $I_j= (h_1,\ldots,h_{j-1}):h_j$. We show that if $I_j\neq S$, then it is generated by linear forms of $S$.

Indeed, let $f\in I_j$. We may assume that the leading term of $f$ has coefficient $1$.  If $h_j\in (h_1,\ldots,h_{j-1})$, then $I_j=S$. We may now assume that $h_j\not\in (h_1,\ldots,h_{j-1})$. Then  $fh_j=\sum_{i=1}^{j-1} g_ih_i$,  where the  $g_i$'s are polynomials in $S$.
This implies that $q=fh^*_j-\sum_{i=1}^{j-1} g_ih^*_i\in J$.
Since $h^*_1<'\cdots<'h^*_j$, we have
\begin{eqnarray}
\label{initial}
\ini_<(q)=\ini_<(f)h^*_j.
\end{eqnarray}
Thus there exists an element $g$ in the Gr\"{o}bner basis of $J$ such that $\ini_<(g)$ divides $\ini_<(f)h^*_j$. By assumption,  $\ini_<(g)=x_{a}^{r}w$ with $0\leq r\leq 1$ and $w$ a monomial in $K[y_1,\ldots,y_m]$.
Suppose $r=0$. Then $w\in \ini_<(J)$ and divides $h^*_j$, contradicting the fact that $h^*_j$ is a standard monomial. Hence $r=1$.

Let $v=h^*_j/w$,  $u=\ini_<(f)/x_a$  and $g=lw -\sum_{i=1}^rp_iw_i$, where $\ini_<(l)=x_a$, $p_i\in S$, $w_i\in K[y_1,\ldots,y_m]$, $\deg w_i=k$ and $w_i<w$ for all $i$. Then $gv=lh_j^*- \sum_{i=1}^r p_i(w_iv)$. Since  $w_iv<h_j^*$ it follows from Lemma \ref{standard} that  $\sum_{i=1}^rp_i(w_iv)\in (h_1^*,\ldots,h_{j-1}^*)$. Therefore, $l\in I_j$ and hence $f'=f-ul\in I_j$. Since $\ini_<(f)=\ini_<(ul)$, it follows that $ \ini_<(f')<\ini_<(f)$. Hence by induction we may assume that $f'$ is a linear combination  of linear forms in $I_j$ with coefficients in $S$. Since $f=f'+ul$ with $l\in I_j$, the desired conclusion follows.
\end{proof}

\medskip
We fix an integer $k\geq 1$. With the assumptions and notation of Theorem~\ref{main},  let $h_1,\ldots,h_s$ be the generators of $A_k$ as described in the proof of Theorem~\ref{main}. Then we have seen that the ideals  $I_j=(h_1,\ldots, h_{j-1}):h_j$ are  generated by linear forms for $j=1,\ldots,s$. The following corollary complements the results of Theorem~\ref{main}.

\begin{Corollary}
\label{maybe}
If $I_j\neq S$, then the ideal  $\ini_<(I_j)$ is generated by all $x_i\in S$  such that $x_ih_j^*\in \ini_<(J)$.
\end{Corollary}

\begin{proof}
Let $l_1,\ldots,l_r$ be the linear forms generating $I_j$. We may choose the generators in a way that $\ini_<(l_s)\neq \ini_<(l_t)$ for $s\neq  t$. Then $l_1,\ldots,l_r$ is the Gr\"{o}bner basis of $I_j$. By (\ref{initial}), for any $l_s\in I_j$, we have $\ini_<(l_s)h_j^*\in \ini_<(J)$.  This shows that any generator of $\ini_<(I_j)$ has the desired form. Conversely, let $g\in J$ with $\ini_<(g)=x_ih_j^*$. Then  $g=lh_j^*+\sum_{t=1}^r g_tv_t$ with $\ini_<(l)=x_i$, $g_t\in S$ and $v_t$ monomials in $K[y_1,\ldots,y_m]$ of degree $k$ with $h_j^*>v_t$ for all $t$. By using Lemma~\ref{standard},  the sum  $\sum_{t=1}^r g_tv_t$  can be rewritten as linear combination of the monomials $h_i^*$ with $h_j^*>h_i^*$. This shows that $l\in I_j$,  and  the desired result follows.
\end{proof}

In the following Proposition we show that under special assumption on the reduced Gr\"{o}bner basis of $J$, the generating set $\{h_1,\ldots,h_s\}$ of $A_k$ described in the proof of Theorem~\ref{main} is a minimal generating set.

\begin{Proposition}
\label{carisokay}
We keep the assumptions of Theorem \ref{main}. Let $\mathcal G$ denote the reduced Gr\"{o}bner basis of $J$. If
$\mathcal G\subseteq (x_1,\ldots,x_n)T$, then $\{h_1,\ldots,h_s\}$ is a minimal generating set of $A_k$.
\end{Proposition}

\begin{proof}
By contradiction assume that $h_j=\sum_{i\neq j} g_ih_i$, for some $j$, where $g_i$'s are homogeneous polynomials in $S$. Then $q=h_j^*-\sum_{i\neq j} g_ih_i^*\in J$. Without loss of generality we may assume that
$q$ has the smallest initial term with respect to $<$ among the elements of $J$ of the form $h_j^*-\sum_{i\neq j} z_ih_i^*$, where $z_i\in S$ are homogeneous.
Since $h_j^*$ is a standard monomial, $\ini_<(q)=\ini_<(g_th_t^*)$ for some $t\neq j$. Hence there exists an element $f\in \mathcal G$ such that $\ini_<(f)$ divides $\ini_<(q)=\ini_{<_x}(g_t)h_t^*$.
Let $f=\sum_{\ell=1}^k f_{\ell}p_{\ell}$, where $f_{\ell}$'s are homogeneous polynomials in $S$ and $p_{\ell}$'s are monomials in $K[y_1,\ldots,y_m]$ and $\ini_<(f)=\ini_{<_x}(f_1)p_1$. Then $\ini_{<_x}(f_1)$ divides $\ini_{<_x}(g_t)$ and
$p_1$ divides $h_t^*$. So $q'=q-(\ini_{<_x}(g_t)/\ini_{<_x}(f_1)) (h_t^*/p_1)f\in J$ and $\ini_<(q')<\ini_<(q)$.
 Hence $q'$ is of the form $h_j^*-\sum_{i=1}^d g_i'w_i$, where $w_i$'s are monomials of degree $k$ in $K[y_1,\ldots,y_m]$ and $g_i'$ are homogenous polynomials in $S$.
 By Lemma \ref{standard}, one can replace each $w_i$ by a linear combination of $h_1^*,\ldots,h_s^*$ with coefficients in $S$.  By our assumption on $\mathcal G$, we have $f_{\ell}\in (x_1,\ldots,x_n)S$ for all ${\ell}$. This together with the fact that $q'$ is bihomogeneous implies that
$h_j^*$ does not appear in this linear combination which is equal to $w_i$. So $q'=h_j^*-\sum_{i\neq j} \lambda_i'h_i^*$, where $\lambda_i\in S$ are homogeneous. But $q'\in J$ and $\ini_<(q')<\ini_<(q)$, which contradicts to the choice of $q$.
\end{proof}

Let $<_x$ be a monomial order on $K[x_1,\ldots,x_n]$, $I$ be a monomial ideal with the minimal system of monomial generators $f_1,\ldots,f_m$,
%such that $f_1<_x\cdots<_xf_d$.
$<'$ be an order on $K[y_1,\ldots,y_m]$ induced by $y_1<'\cdots<'y_m$ and $<$ be a monomial order on $T$ defined as in (\ref{monomialorder}). Then we have

%We say that $I$ satisfies the $x$-condition, when the defining ideal of the Rees ring $\mathcal{R}(I)$ satisfies the $x$-condition with respect to $<$.

\begin{Theorem}
\label{monomialcase}
Let $I$ be a monomial ideal and $J$ be the defining ideal of the Rees ring $\mathcal{R}(I)$. If $\ini_{<}(J)$ is generated by quadratic monomials, then all powers of $I$ have linear quotients with respect to their minimal generating sets and hence they are componentwise linear.
\end{Theorem}

\begin{proof}
Let $f_1,\ldots,f_m$ be the minimal system of monomial generators of $I$ such that $f_1<_x\cdots<_xf_m$.
Let $k\geq 1$ be an integer.  For any monomial $h=f_{i_1}\cdots f_{i_k}\in I^k$, we set
$h^{*}=y_{i_1}\cdots y_{i_k}$. Let $h_1^{*},\ldots,h_{d}^*$ be the standard monomials of degree $k$ in $T$.  Then by Lemma~\ref{standard}, $h_1,\ldots,h_d$ is a system of generators for $I^k$. We claim that this is indeed a minimal system of generators. Note that since $f_1,\ldots,f_m$ is a minimal generating set of $I$, $y_1,\ldots,y_m$ are standard. Let $k>1$
and by contradiction suppose that $h_j=uh_i$ for some $i\neq j$. If $u=1$, then $h_i^*-h_j^*\in J$, which implies that $h_i^*\in\ini_{<}(J)$ or $h_j^*\in\ini_{<}(J)$, a contradiction. So $u\neq 1$.  Without loss of generality we may assume that $h_i^*$ is the smallest standard monomial with respect to $<'$ such that $h_j=uh_i$ for some monomial $u\neq 1$.
%First let $k=1$. Then we may assume that $h_j=y_j$ and $h_i=y_i$. Then $f_j=uf_i$. So
%$f_j>_{lex}f_i$ which implies that $y_j>y_i$. Thus $y_j=\ini(y_j-uy_i)\in \ini(J)$, a contradiction. Now let $k>1$.
We have $q=h_j^*-uh_i^*\in J$. Since $h_j^*\notin \ini(J)$,   $\ini_<(q)=uh_i^*$.  By assumption there exists a binomial $g$ in the Gr\"{o}bner basis of $J$ such that $\ini_<(g)$ is generated by a quadratic monomial and $\ini_<(g)$ divides $uh_i^*$. If $\ini_<(g)=y_ry_t$ for some $r$ and $t$, then $y_ry_t$ divides $h_i^*$, which implies that $h_i^*\in \ini_<(J)$, a contradiction.
So $\ini_<(g)=x_ry_t$ for some $r$ and $t$ and then $x_r$ divides $u$ and $y_t$ divides $h_i^*$. Let $g=x_ry_t-u'y_{p}$ for some monomial $u'\in S$. Hence $x_rf_t=u'f_p$. We should have $u'\neq 1$, otherwise $f_p$ is a multiple of $f_t$, a contradiction.
We have $uh_i=(u/x_r)x_rf_t(h_i/f_t)=(u/x_r)u'f_p(h_i/f_t)$. Hence $h_j=wh'$, where $w=(u/x_r)u'\neq 1$ and $h'=f_p(h_i/f_t)$. So $h_j^*-w(h')^*\in J$ and $h'\neq h_j$ which implies that $(h')^*\neq h_j^*$. Also $y_p<'y_t$ implies that $(h')^*<'h_i^*$.
By the assumption on $h_i^*$, the monomial $(h')^*$ is not standard. Then $(h')^*-u''(h'')^*\in J$, for some standard monomial  $(h'')^*$ with $(h'')^*<'(h')^*$ and some monomial $u''\in S$. Hence $h_j=wh'=wu''h''$ and then
$h_j^*-wu''(h'')^*\in J$. We have $(h'')^*<'(h')^*<'h_i^*$ and $(h'')^*\neq h_j^*$, because $wu''\neq 1$. This contradicts to the choice of $h_i^*$. So the claim is proved

Now, by Theorem \ref{main}, for each $k$, $I^k$ has linear quotients with respect to its minimal set of monomial generators. Then by ~\cite[Theorem 8.2.15]{HH}, $I^k$ is componentwise linear.
\end{proof}

For the next result we put an  extra condition on the monomial order $<'$  defined on  $K[y_1,\ldots,y_m]$.

We keep the assumptions that $\deg f_i =d_i$ and that $d_1\leq d_2\leq \cdots \leq d_m$,  and define the weight vector $\wb=(d_1,\ldots,d_m)$. We denote by $<^*$  an order on $K[y_1,\ldots,y_m]$  with $y_1<^*\cdots <^* y_m$, and  define  $<'$  as follows:
\begin{eqnarray}
\label{weight}
\yb^{\ab}<' \yb^{\bb}  \quad  \text{if} \quad  \ab\cdot \wb<\bb\cdot\wb, \quad \text{or else} \quad \ab\cdot \wb=\bb\cdot\wb \quad \text{and} \quad \yb^{\ab}<^* \yb^{\bb}.
\end{eqnarray}
Here for a vector $\ab=(a_1,\ldots,a_m)$, $\yb^{\ab}=y_1^{a_1}\cdots y_m^{a_m}$ and $\cb\cdot \db=\sum_{i=1}^mc_id_i$ is the standard scalar product.
\begin{Corollary}
\label{weightedworks}
With the  monomial order $<'$ on $K[y_1,\ldots,y_m]$ defined in {\em (\ref{weight})}, let $<$ be the monomial order on $T$ defined in
{\em (\ref{monomialorder})}, and assume that the defining ideal $J$ of $A$ satisfies the $x$-condition with respect $<$.  Then $A_k$ is componentwise linear for all $k\geq 1$,
%\begin{enumerate}
%\item[{\em (a)}]A_k$ is componentwise linear
%\item[{\em (b)}] $\depth A_{k}\geq \depth A_{k+1}$ .
%\end{enumerate}
\end{Corollary}

\begin{proof}
 In the proof of Theorem~\ref{main}, we noticed that  $A_k$ has a system of generators $h_1,\ldots,h_s$, where each of them is of the form $f_{i_1}\cdots f_{i_k}$ and such that $h^*=y_{i_1}\cdots y_{i_k}$ is a standard monomial of $T$ with respect to $<$. Moreover, if choose the labeling of the $h_i$ such  $h^*_1<'\cdots<'h^*_s$, then $h_1,\ldots,h_s$ has linear quotients.   The monomial order defined in (\ref{weight}) implies that $\deg h_1\leq \deg h_2 \leq \cdots \leq \deg h_s$. Thus we may apply Proposition~\ref{componentwise}, and the desired conclusion follows. %As before, we set $I_j=(h_1,\ldots,h_{j-1}):h_j$  for $j=1, \ldots,s$.
\qed
\end{proof}

%(b) For $A_{k+1}$ we choose a system of generators $g_1,\ldots,g_t$ in the  same way as for $A_k$. In particular,  $g^*_1<'\cdots<'g^*_t$ and %$g_1,\ldots,g_t$ is a sequence with linear quotients. We set $I_j'=(g_1, \ldots,g_{j-1}):g_j$ for $j=1,\ldots,t$, and claim that for each $j$ there exists $l$ such that $I_l\subseteq I_j'$. This claim together part (a)  and  Proposition~\ref{generalbetti} implies that $\projdim A_k\leq \projdim A_{k+1}$, which then yields  that $\depth A_{k+1}\leq \depth A_k$, as desired.

%Proof of the claim: Say,  $g_j=f_{i_1}f_{i_2}\cdots f_{i_{k+1}}$ where $y_{i_1}\cdots y_{i_{k+1}}$ is a standard monomial. Then $y_{i_1}\cdots y_{i_{k+1}}\not\in\ini_<(J)$, and hence $y_{i_1}\cdots y_{i_{k}}\not\in\ini_<(J)$. This shows that $y_{i_1}\cdots y_{i_{k}}$ is a standard monomial. Therefore,  $f_{i_1}f_{i_2}\cdots f_{i_{k}}=h_l$ for some $l$. We claim that $I_l\subseteq I'_j$. Indeed, let $f\in I_l$. Then $fh_l\in(h_1,\ldots,h_{l-1})$. It follows that $fg_j\in(h_1 f_{i_{k+1}},\ldots, h_{l-1}f_{i_{k+1}})$. Thus it suffices to prove that $h_sf_{i_{k+1}}\in (g_1,\ldots,g_{j-1})$ for $s=1,\ldots,l$.  Indeed, since $h_s^*<'h_l^*$ it follows that $h_s^*y_{i_{k+1}}<'h_l^*y_{i_{k+1}}=g_j^*$.  Now $h_s^*y_{i_{k+1}}$ is a linear combination of standard monomials $y_{r_1}\cdots y_{r_{k+1}}
%\leq' h_s^*y_{i_{k+1}}<'g_j^*$. Thus each of the monomials $y_{r_1}\cdots y_{r_{k+1}}$ belongs to $(g_1^*,\ldots,g_{j-1}^*)$,  and this   implies that $h_sf_{i_{k+1}}\in (g_1,\ldots,g_{j-1})$ and  completes the proof.

\section{Symmetric powers of modules}

Let $M$ be a finitely generated graded $S$-module with
homogeneous generators $f_1, \ldots,f_m$.  Furthermore,   let  $F=\bigoplus_{i=1}^m Se_i$  be the free $S$-module with basis $e_1,\ldots,e_m$ and   $\varphi:F\rightarrow M$ the  $S$-module homomorphism with $\varphi(e_i)=f_i$ and $C=\Ker(\varphi)$. Let $g_1,\dots, g_r$ be a system of homogeneous generators of $C$ with
\[
g_i=\sum_{j=1}^ma_{ji}e_j
\]
for $i=1,\ldots,r$. As before, we let $T=S[y_1,\ldots,y_m]$. Then,  as a graded $K$-algebra,  the symmetric algebra $\Sym(M)$ of $M$ is isomorphic to $T/J$, where $J$ is generated by the polynomials
\[
h_i=\sum_{j=1}^ma_{ji}y_j
\]
for $i=1,\ldots,r$.

The graded components $\Sym_j(M)$  of $\Sym(M)$ are called the {\em symmetric powers} of $M$. Let  $I$ be a graded ideal and $\mathcal{R}(I)$ be the Rees ring of $I$.   There is a canonical epimorphism $\Sym(I)\to \mathcal{R}(I)$.   The ideal $I$ is called of {\em linear type} if this epimorphism is an isomorphism. In this case the ordinary powers and the symmetric powers of $I$ coincide.

\medskip
As a first application of the theory developed in Section~\ref{one} we study a special class of modules and their symmetric powers. For this purpose we fix a finite simple graph $G$ on the vertex set $[n]$, consider the free $S$-module $F$ with basis $e_1,\ldots,e_n$ and let $C_G\subset F$ be the submodule of $F$ generated  by the elements
\[
x_ie_j-x_je_i \quad\text{with} \quad \{i,j\}\in E(G).
\]

We call the module $M_G=F/C_G$ the {\em edge module} of $G$.  Then for $i=1,\ldots,n$ the elements $f_i=e_i+C_G$ generate $M_G$. For example, if $G=K_n$ is the complete graph on $[n]$, then $C_G$ is the first syzygy module of the graded maximal ideal $\mm=(x_1,\ldots, x_n)$, and hence  $M_{G}\iso  \mm$.

Recall that a graph $G$ is called \emph{chordal}, if it has no induced cycle of length $k\geq 4$.
Also for a graph $G$ and $W\subset V(G)$, the induced subgraph $G_W$ of $G$ is the graph with vertex set $W$ and edge set $E(G_W)=\{\{i,j\}\in E(G)\:\; i,j\in W\}$ .

\begin{Theorem}
\label{mg}
Let $G$ be a finite simple graph on $[n]$. The following conditions are equivalent:
\begin{enumerate}
\item[{\em (i)}] The defining ideal of $\Sym(M_G)$ satisfies the $x$-condition.
\item[{\em (ii)}] $\Sym_j(M_G)$ has linear resolution for all $j$.
\item[{\em (iii)}] $M_G$ has linear resolution.
\item[{\em (iv)}] $G$ is chordal.
\end{enumerate}
\end{Theorem}

\begin{proof}
(i) \implies (ii)  follows from Theorem~\ref{exactlythesame}, and (ii)\implies (iii) is trivial.

(iii) \implies (iv): We define a $\ZZ^n$-grading on $M_G$. In order to simplify notation we write  $\deg f=\xb^{\ab}$ when $f$ has $\ZZ^n$-degree  $\ab\in \ZZ^n$. By using this notation, we set $\deg e_i=x_i$   for $i=1,\ldots,n$. Then the generators $c_{ij}=x_ie_j-x_je_i$ of $C_G$ are of degree $x_ix_j$. Thus $M_G$ is naturally $\ZZ^n$-graded and hence  admits a  minimal $\ZZ^n$-graded free $S$-resolution $(\FF, \partial)$.

Suppose $G$ is not chordal. Then for some $r$ with $4\leq  r\leq n$ there exists an induced $r$-cycle $C_r$ of $G$. For simplicity we may assume that $[r]$ is the vertex set of $C_r$. Let $F_i=\Dirsum_{j=1}^{\beta_i}S(-\ab_{ij})$. We define
\[
H_i=\Dirsum_{j,\;\supp(\xb^{\ab_{ij}})\subseteq [r]}S(-\ab_{ij}).
\]
Restricting for each $i$ the chain maps $\partial_i$ of $\FF$ to $H_i$ we obtain an acyclic subcomplex $\HH$ of $\FF$ with $H_0(\HH)=M_{C_r}$.
We will show that the $\ZZ$-graded free resolution of $M_{C_r}$  is of the form
\begin{eqnarray*}
\label{resolution}
  0\to S(-r)\stackrel{\varphi_2}{\To} S^r(-2)\stackrel{\varphi_1}{\To} S^r(-1)\To M_{C_r}\to 0.
\end{eqnarray*}
In particular, $M_{C_r}$ does not have linear resolution. Therefore $M_G$ does not have linear resolution, a contradiction.

The map  $\varphi_1$ is given by the matrix
\[
A=
\begin{pmatrix}
-x_2& 0 &\ldots &\ldots &0& x_r\\
x_1 & -x_3 & 0 &&\vdots& 0 \\
0 & x_2 & \ddots&&\vdots& \vdots \\
\vdots & &\ddots& \ddots&0&\vdots\\
\vdots &\ldots &&\ddots&-x_r & 0\\
0& \ldots&&0& x_{r-1} &-x_1\\
\end{pmatrix}.
\]
The columns of this matrix correspond to the generating elements of $C_{G'}$ where $G'=C_r$. We define $\varphi_2$ by the column vector
\[
(u_1, u_2, \ldots,  u_r)^{\sf T},
\]
where  $u_i =(x_1x_2\cdots x_n) /(x_ix_{i+1})$ for $i=1,\ldots, r-1$,  and  $u_r=(x_1x_2\cdots x_n) /(x_1x_{r})$. With the definitions given,
\begin{eqnarray*}
\label{resolution}
\GG\:\;  0\to S(-r)\stackrel{\varphi_2}{\To} S^r(-2)\stackrel{\varphi_1}{\To} S^r(-1)\to 0
\end{eqnarray*}
is a complex with $H_0(\GG)=M_{C_r}$.

It remains to be shown that $\GG$ is acyclic. To  show this we apply the  Buchsbaum-Eisenbud acyclicity criterion \cite{BE} (see also \cite[Theorem 9.1.6]{BH}). For a module homomorphism $\varphi\: F\to G$ between finitely generated free $S$-modules which is represented by a  matrix $B$, we denote by $I_t(\varphi)$ the ideal generated by the $t$-minors of $B$. This definition does not depend on the choice of the matrix representing $\varphi$. Now the Buchsbaum-Eisenbud acyclicity criterion applied to our situation says that $\GG$ is acyclic if $\grade(I_{r-1}( \varphi_1))\geq 1$ and $\grade(I_{1}( \varphi_2))\geq 2$. We see that $\grade(I_{r-1}( \varphi_1))\geq 1$, because
$x_1\cdots x_{r-1}\in I_{r-1}( \varphi_1)$, so that $I_{r-1}( \varphi_1)\neq 0$.  Since $I_{1}( \varphi_2)$ is the ideal generated by $u_1,\ldots,u_r$ and since $\gcd(u_1,\ldots,u_r)=1$, it follows that $\grade(I_{1}( \varphi_2))\geq 2$.

(iv) \implies (i):  Note that $\Sym(M_G)\iso T/J$, where $T=K[x_1,\ldots,x_n,y_1,\ldots,y_n]$ and $$J=(x_iy_j-x_jy_i:\; \{i,j\}\in E(G))$$ is the binomial edge  ideal  of $G$, as introduced in \cite{HHHKR} and \cite{Oht}. We choose the lexicographic order induced by $x_1>x_2>\cdots >x_n>y_1>y_2> \cdots >y_n$ and show that $\ini_<(J)$ is generated by monomials of $x$-degree $\leq 1$.  The Gr\"obner basis of $J$ is known by \cite{HHHKR} and \cite{Oht}. Here we refer to \cite[Theorem 2.1]{HHHKR} or  \cite[Theorem 7.11]{HHO} and recall the following:
let $G$ be a simple graph on $[n]$,
and let $i$ and $j$ be two vertices of $G$ with $i<j$.
A path $i=i_0,i_1,\ldots,i_r=j$ from $i$ to $j$
is called \emph{admissible}, if
\begin{enumerate}
\item[(a)] $i_k\neq i_\ell$  for $k\neq \ell$;
\item[(b)] for each $k=1,\ldots,r-1$ one has either $i_k<i$ or $i_k>j$;
\item[(c)] for any proper subset $\{j_1,\ldots,j_s\}$
of $\{i_1,\ldots,i_{r-1}\}$, the sequence $i,j_1,\ldots,j_s,j$
is not a path.
\end{enumerate}
Given an admissible path
\[
\pi: i=i_0,i_1,\ldots,i_r=j
\]
from $i$ to $j$, where $i < j$, we associate the monomial
\[
u_{\pi}=(\prod_{i_k>j}x_{i_k}) (\prod_{i_\ell<i}y_{i_\ell}).
\]
Then the set of binomials
\[
{\mathcal G}
= \Union_{i<j} \,
\{\,u_{\pi}f_{ij}\,:\;\text{$\pi$ is an admissible path from $i$ to $j$}\,\}
\]
is the reduced Gr\"obner basis of $J$ with respect to $<$, where $f_{ij}=x_iy_j-x_jy_i$. In view of this result we must show that the following condition $(*)$ is satisfied:  if $i<j$ and $\pi: i=i_0,i_1,\ldots,i_r=j$  is an admissible path from $i$ to $j$, then $i_k<i$ for all $k=1,\ldots, r-1$.

Now we use that $G$ is a chordal graph, which by a theorem by Dirac \cite{} implies that $G$
has a perfect elimination ordering. In other words, there is a labeling  $v_1,\ldots,v_n$ on the vertices of $G$ such that each $v_j$ is a simplicial vertex of the induced subgraph $G_j=G_{\{v_1,\ldots,v_j\}}$. By, definition $v_j$ is simplicial on $G_j$,  if the induced subgraph of $G$ on the vertex set $\{v_{k}\:\; 1\leq k<j, \{v_k,v_j\}\in E(G)\}$ is a complete graph. By renaming the vertices we may assume that $v_i=i$ for $i=1,\ldots,n$.

Let $i<j$ and $\pi: i=i_0,i_1,\ldots,i_r=j$ be an admissible path between $i$ and $j$. We prove $(*)$ by induction on $n$. First assume that $j<n$.  If $i_k=n$ for some $k$ with $1\leq k<r$, then  $\{i_{k-1},n\}$ and $\{i_{k+1}, n\}$ are edges of $G$ which implies that $\{i_{k-1}, i_{k+1}\}$ is an edge of $G$, since $n$ is a simplicial vertex of $G$. This contradicts condition (c) for admissible  paths. So $\pi$ is an admissible path in the induced subgraph $G_{[n-1]}$.  Since $G_{[n-1]}$ is chordal (by Dirac), by induction hypothesis condition $(*)$ is satisfied when $j<n$. On the other hand if $j=n$, then $i_k<n$ for $k=1,\ldots,r-1$, and since  $\pi$ is admissible, it follows that $i_k<i$ for $k=1,\ldots,r-1$, as desired.
\end{proof}

The following result gives some additional information of the symmetric powers of the modules $M_G$.

\begin{Theorem}
\label{additional}
Let $G$ be a chordal graph on $[n]$. We may assume that the vertices of $G$ are labeled such that for $i=1,\ldots,n$, the vertex $i$ is a simplicial vertex of the induced subgraph $G_{[i]}$ for all $i$.  Then
\begin{enumerate}
\item[{\em (a)}] $\depth M_G\geq n- \max\{|N_{G_{[j]}}(j)|\:\; j=1,\ldots,n\}$.

\item[{\em (b)}] $\depth \Sym_k(M_G)\geq \depth \Sym_{k+1}(M_G)\geq 1$ for all $k\geq 1$.

\item[{\em (c)}] $1\leq \lim_{k\to \infty} \depth \Sym_k(M_G)\leq \max\{|A|-c(A)\:\; A\subset [n]\}$, where $c(A)$ denotes number of connected components of $G_{[n]\setminus A}$.
\end{enumerate}
\end{Theorem}

\begin{proof}
In the proof of Theorem~\ref{mg} we have seen  that $\Sym(M_G)$ is naturally $\ZZ^n$-graded. The way how the $\ZZ^n$-grading is defined it follows  that all products $f_{i_1}\cdots f_{i_k}$ are homogeneous with respect to this grading. This implies that all the colon ideals which we consider in the symmetric powers of $M_G$ are monomial ideals. We will use this fact in the proof of (a) and (b).

Proof of (a):   $\Sym(M_G)=T/J$,  where $T=K[x_1,\ldots,x_n,y_1,\ldots,y_n]$ and $J=(x_iy_j-x_jy_i\:\; \{i,j\}\in E(G))$.  Let $f_1,\ldots, f_n$ be the generators of $M_G$ corresponding  to  $y_1,\ldots, y_n$. Then, as shown in the proof of Theorem ~\ref{mg}, we have $x_iy_j\in \ini_<(J)$ if only if $i<j$ and $\{i,j\}\in E(G)$. Thus, $\ini_<(C_G)= \Dirsum_{i=1}^nI_je_j$, where $I_j=(x_i\:\; i\in N_{G_{[j]}(j)})$. It follows that $\projdim F/\ini_<(C_G)=\max\{|N_{G_{[j]}}(j)|\:\; j=1,\ldots,n\}$. By using the well-known upper semicontinuity of graded Betti numbers we obtain
\begin{eqnarray*}
\depth M_G&=& n-\projdim F/C_G \geq n-\projdim F/\ini_<(C_G)\\
& =&n-\max\{|N_{G_{[j]}}(j)|\:\; j=1,\ldots,n\}.
\end{eqnarray*}

Proof of (b): It follows from the description of the Gr\"obner basis of $J$ as given in the proof of Theorem~\ref{mg} that (i) $\ini_<(J)\subseteq (x_1,\ldots,x_{n-1})$, and (ii) $\ini_<(J)\subseteq (y_2,\ldots ,y_n)$. Now (i) implies that $x_n$  is  non-zerodivisor on $T/\ini_<(J)$, and hence also a non-zerodivisor on $T/J=\Sym(M_G)$.  Thus the multiplication map $$\mu_{x_n}\: \Sym(M_G)\to \Sym(M_G)$$ is injective. For each $k\geq 0$, $\mu_{x_n}(\Sym_k(M_G))\subseteq \Sym_k(M_G)$. Hence $\mu_{x_n}$ restricts to the multiplication map
$\mu_{x_n}\:\Sym_k(M_G)\to \Sym_k(M_G)$ which is also injective. This shows that $x_n$ is non-zerodivisor for each of the $S$-modules $\Sym_k(M_G)$, and hence  $\depth \Sym_k(M_G)\geq 1$ for all $k$.

In order to simplify notation we set $N_k=\Sym_k(M_G)$. It remains to be shown that $\depth N_k\geq \depth N_{k+1}$. To show this we use  (ii) which implies   that $y_1$ is a non-zerovisor on $T/J=\Sym(M_G)$.
In other words, the $S$-linear map
\[
N_k\to N_{k+1}, \quad n\mapsto f_1n
\]
is injective for all $k$.  We may assign to $T$ the standard grading. Then each $N_k$ is generated in degree $k$, and by Theorem~\ref{mg}, $N_k$ has $k$-linear resolution. The short exact sequence of graded $S$-modules
\[
0\to N_k(-1)\to N_{k+1}\to N_{k+1}/f_1N_k\to 0
\]
induces the long exact sequence
\begin{eqnarray*}
&&\cdots \To \Tor_{i+1}(K,N_{k+1}/f_1N_k)_{i+1+k}\To\\
 &&\Tor_{i}(K,N_k(-1))_{i+(k+1)}=
\Tor_{i}(K,N_k)_{i+k}\To\Tor_{i}(K, N_{k+1})_{i+(k+1)}\To\cdots
\end{eqnarray*}

Since  the initial degree of $N_{k+1}/f_1N_k$ is bigger than $k$, it follows that
$$\Tor_{i+1}(K,N_{k+1}/f_1N_k)_{i+1+k}=0.$$ Hence $\Tor_{i}(K,N_k)_{i+k}\To\Tor_{i}(K, N_{k+1})_{i+(k+1)}$ is injective.
This shows that $\beta_i(N_k)\leq \beta_i(N_{k+1})$ for all $i\geq 0$ and implies that $\depth N_k\geq \depth N_{k+1}$.

Proof of (c):  By (b),  $\lim_{k\to \infty} \depth \Sym_k(M_G)\geq 1$. On the other hand, by \cite[Theorem 1.1]{HH1} we have
\begin{eqnarray}
\label{lim}
\lim_{k\to \infty} \depth \Sym_k(M_G)\leq \dim \Sym(M_G)-\dim  \Sym(M_G)/\mm \Sym(M_G),
\end{eqnarray}
where $\mm=(x_1,\ldots,x_n)$.
Since the defining ideal $J$ of $\Sym(M_G)$ is the binomial edge ideal of the graph $G$, it follows from \cite[Corollary 3.3]{HHHKR} or \cite[Corollary 7.17]{HHO} that $\dim \Sym(M_G)= \max\{n-|A|+c(A)\:\; A\subset [n]\}$. Thus,  since $\Sym(M_G)/\mm \Sym(M_G)\iso K[y_1,\ldots,y_n]$, the desired result follows from (\ref{lim}).
\end{proof}

\begin{Remark}{\em
As the proof shows, the inequality $$\lim_{k\to \infty} \depth \Sym_k(M_G)\leq \max\{|A|-c(A)\:\; A\subset [n]\}$$ is valid for any graph.
If $\Sym(M_G)$ is Cohen--Macaulay, then  equality holds. This follows from \cite[Theorem 1.1]{HH1}. It would be interesting to express in general this limit depth in terms of $G$. }
\end{Remark}

We conclude this section with the following observation about symmetric powers.

\begin{Proposition}
\label{observation}
Let $M$ be a finitely generated graded $S$-module and let  $\Sym(M)=T/J$ be a minimal presentation of the symmetric power algebra of $M$. If $J$ satisfies the $x$-condition with respect to a monomial order as in {\em (\ref{monomialorder})}, then all symmetric powers $\Sym_k(M)$ of $M$ are componentwise linear.
\end{Proposition}

\begin{proof}
Let $f_1,\ldots,f_m$ be a minimal set of generators of $M$. Then $\Sym(M)=T/J$,  where is generated by elements of the form $\sum_{i=1}^m g_iy_i$ with $\sum_{i=1}^m g_if_i=0$. Since $f_1,\ldots,f_m$ is a minimal set of generators of $M$ is follows that all $g_i$ belong to $(x_1,\ldots,x_n)$. Hence by Proposition~\ref{carisokay}, for each $k$,  the generators $h_1,\ldots,h_d$ of $\Sym_k(M)$ given in the proof of Theorem~\ref{main} which is a linear quotient sequence is also a  minimal set of generators of $\Sym_k(M)$. Now the desired result follows from \cite[Theorem 8.2.15]{HH} in the  version for modules.
\end{proof}

\section{Families of ideals whose powers are componentwise linear}

In this section we apply Theorem~\ref{main} and Theorem~\ref{monomialcase} to   the  vertex cover ideals of certain  families of graphs.   We show that the vertex cover ideals of these graphs satisfy the $x$-condition. This then implies that the powers of such  vertex cover ideals  are componentwise linear.

Let $G$ be a finite simple graph with the vertex set $V(G)$ and the edge set $E(G)$. A subset $C\subseteq V(G)$ is called a \emph{vertex cover} of $G$,
when it intersects any edge of $G$. Moreover, $C$ is called a \emph{minimal vertex cover} of $G$, if it is a vertex cover and no proper subset of $C$ is a vertex cover of $G$.
Let $C_1,\ldots,C_m$  be the minimal vertex covers of the graph $G$. The vertex cover ideal of $G$ is defined as
$I_G=(x_{C_1},\ldots,x_{C_m})$, where $x_{C_j}=\prod_{x_i\in C_j} x_i$.

\medskip
{\em Biclique graphs.}
Let $V = \{x_1, \ldots, x_p, y_1, \ldots, y_q, z_1, \ldots, z_r\}$ be a finite set.  We write $G_{p,q}$ for the complete graph on $\{x_1, \ldots, x_p, y_1, \ldots, y_q\}$ and $G_{p,r}$ for the complete graph on $\{x_1, \ldots, x_p, z_1, \ldots, z_r\}$.  We study the vertex cover ideal of the finite simple graph
\[
\Gamma_{p,q,r} = G_{p,q} \cup G_{p,r}
\]
on $V$.  For example, the finite simple graph $\Gamma_{2,3,2} = G_{2,3} \cup G_{2,2}$ is drawn in Figure~\ref{figure1}.

%%%%%%%%%%%%%%%%%%%%%%%%%%%%%%
\begin{figure}[h]
\begin{center}
\includegraphics[height=4cm]{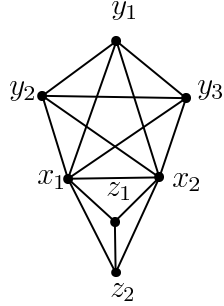}
\end{center}
\caption{The graph $G =\Gamma_{2,3,2} $}
\label{figure1}
\end{figure}
%%%%%%%%%%%%%%%%%%%%%%%%%%%%%%

\noindent
Let $S = K[x_1, \ldots, x_p, y_1, \ldots, y_q, z_1, \ldots, z_r]$ denote the polynomial ring in $p+q+r$ variables over a field $K$.  Let
\[
\xb = x_1 \cdots x_p, \, \yb = y_1 \cdots y_q, \, \zb = z_1 \cdots z_r.
\]
It follows that the vertex cover ideal $I_{\Gamma_{p,q,r}} \subset S$ of $\Gamma_{p,q,r}$ is generated by the monomials
\begin{eqnarray*}
g_i &=& (\xb/x_i) \yb \zb, \, \, \, \, \, 1 \leq i \leq p, \\
f_j^k &=& \xb (\yb/y_j) (\zb/z_k), \, \, \, \, \, 1 \leq j \leq q, 1 \leq k \leq r.
\end{eqnarray*}
Let $\Rc(I_{\Gamma_{p,q,r}})$ denote the Rees algebra of $I_{\Gamma_{p,q,r}}$.  Thus $\Rc(I_{\Gamma_{p,q,r}})$ is the subring of $S[t]$ which is generated by the variables and monomials
\begin{eqnarray*}
x_i, \, \, y_j, \, \, z_k, \, \, g_i t, \, \, f_j^k t, \, \, \, \, \, 1 \leq i \leq p, \, 1 \leq j \leq q, \, 1 \leq k \leq r.
\end{eqnarray*}
Let $T = K[\{x_i, y_j, z_k, \varphi_i, \psi_j^k : 1 \leq i \leq p, \, 1 \leq j \leq q, \,1 \leq k \leq r\}]$ denote the polynomial ring over $K$ in $2p+(q+1)(r+1)-1$ variables and define the surjective ring homomorphism $\pi : T \to \Rc(I_{\Gamma_{p,q,r}})$ by setting
\[
\pi(x_i) = x_i, \pi(y_j) = y_j, \pi(z_k) = z_k, \pi(\varphi_i) = g_{i}t, \pi(\psi_j^k) = f_j^k t.
\]
Let $J_{\Gamma_{p,q,r}} \subset T$ denote the toric ideal of $\Rc(I_{\Gamma_{p,q,r}})$.  In other words, $J_{\Gamma_{p,q,r}}$ is the kernel of $\pi$.

Our mission is to compute a certain Gr\"obner basis of $J_{\Gamma_{p,q,r}}$.  Let $<_{\rm lex}$ denote the lexicographic order on $K[\{\varphi_i, \psi_j^k : 1 \leq i \leq p, \, 1 \leq j \leq q, \,1 \leq k \leq r\}]$ induced by the total order of the variables as follows:
\[
\varphi_1 > \cdots > \varphi_p > \psi_1^r > \psi_1^{r-1} > \cdots > \psi_1^1 > \psi_2^r > \cdots > \psi_2^1 > \cdots > \psi_q^r > \cdots > \psi_q^1
\]
and $<'_{\rm lex}$ the lexicographic order on $K[\{x_i, y_j, z_k : 1 \leq i \leq p, \, 1 \leq j \leq q, \,1 \leq k \leq r\}]$ induced by the total order of the variables as follows:
\[
y_q > \cdots > y_1 > z_1 > \cdots > z_r > x_p > \cdots > x_1.
\]
Furthermore, we introduce the monomial order $<$ on $T$ by setting that
\begin{eqnarray*}
&   & \prod_{i=1}^p x_i^{\alpha(x_i)} \prod_{j=1}^q y_j^{\alpha(y_j)} \prod_{k=1}^r z_k^{\alpha(z_k)} \prod_{i=1}^p \varphi_i^{\alpha(\varphi_i)} \prod_{j=1}^q \prod_{k=1}^r (\psi_{j}^{k})^{\alpha(\psi_{j}^{k})} \\
& < &
\prod_{i=1}^p x_i^{\beta(x_i)} \prod_{j=1}^q y_j^{\beta(y_j)} \prod_{k=1}^r z_k^{\beta(z_k)} \prod_{i=1}^p \varphi_i^{\beta(\varphi_i)} \prod_{j=1}^q \prod_{k=1}^r (\psi_{j}^{k})^{\beta(\psi_{j}^{k})}
\end{eqnarray*}
if and only if either
\[
\prod_{i=1}^p \varphi_i^{\alpha(\varphi_i)} \prod_{j=1}^q \prod_{k=1}^r (\psi_{j}^{k})^{\alpha(\psi_{j}^{k})}
<_{\rm lex}
\prod_{i=1}^p \varphi_i^{\beta(\varphi_i)} \prod_{j=1}^q \prod_{k=1}^r (\psi_{j}^{k})^{\beta(\psi_{j}^{k})}
\]
or
\[
\prod_{i=1}^p \varphi_i^{\alpha(\varphi_i)} \prod_{j=1}^q \prod_{k=1}^r (\psi_{j}^{k})^{\alpha(\psi_{j}^{k})}
=
\prod_{i=1}^p \varphi_i^{\beta(\varphi_i)} \prod_{j=1}^q \prod_{k=1}^r (\psi_{j}^{k})^{\beta(\psi_{j}^{k})},
\]
\[
\prod_{i=1}^p x_i^{\alpha(x_i)} \prod_{j=1}^q y_j^{\alpha(y_j)} \prod_{k=1}^r z_k^{\alpha(z_k)}
<'_{\rm lex}
\prod_{i=1}^p x_i^{\beta(x_i)} \prod_{j=1}^q y_j^{\beta(y_j)} \prod_{k=1}^r z_k^{\beta(z_k)}.
\]

\begin{Theorem}
\label{GB}
The reduced Gr\"obner basis of the toric ideal $J_{\Gamma_{p,q,r}}$ with respect to the monomial order $<$ on $T$ defined above, consists of the following binomials:
\begin{itemize}
\item
$x_i \varphi_i - y_q z_1 \psi_q^1, \, \, \, \, \, 1 \leq i \leq p$;
\item
$y_j \psi_j^k - y_q \psi_q^k, \, \, \, \, \, 1 \leq j < q, \, 1 \leq k \leq r$;
\item
$z_k \psi_j^k - z_1 \psi_j^1, \, \, \, \, \, 1 \leq j \leq q, \, 1 < k \leq r$;
\item
$\psi_{j_1}^{k_2} \psi_{j_2}^{k_1} - \psi_{j_1}^{k_1} \psi_{j_2}^{k_2}, \, \, \, \, \, 1 \leq j_1 < j_2 \leq q, \, 1 \leq k_1 < k_2 \leq r$.
\end{itemize}
Its initial ideal ${\rm in}_{<}(J_{p,q,r})$ is generated by the quadratic monomials as follows:
\begin{itemize}
\item
$x_i \varphi_i, \, \, \, \, \, 1 \leq i \leq p$;
\item
$y_j \psi_j^k, \, \, \, \, \, 1 \leq j < q, \, 1 \leq k \leq r$;
\item
$z_k \psi_j^k, \, \, \, \, \, 1 \leq j \leq q, \, 1 < k \leq r$;
\item
$\psi_{j_1}^{k_2} \psi_{j_2}^{k_1}, \, \, \, \, \, 1 \leq j_1 < j_2 \leq q, \, 1 \leq k_1 < k_2 \leq r$.
\end{itemize}
\end{Theorem}

\begin{proof}
Let $\Omega$ denote the set of those monomials $u \in T$ for which none of the monomials $x_i \varphi_i, \, y_j \psi_j^k, \, z_k \psi_j^k, \, \psi_{j_1}^{k_2} \psi_{j_2}^{k_1}$ listed above divides $u$.  By virtue of \cite[Lemma 1.1]{AHH} our task is to show that, for $u \in \Omega$ and $v \in \Omega$ with $u \neq v$, one has $\pi(u) \neq \pi(v)$.  Let
\begin{eqnarray*}
u & = & \prod_{i=1}^p x_i^{\alpha(x_i)} \prod_{j=1}^q y_j^{\alpha(y_j)} \prod_{k=1}^r z_k^{\alpha(z_k)} \prod_{i=1}^p \varphi_i^{\alpha(\varphi_i)} (\psi_{j_1}^{k_1}\psi_{j_2}^{k_2}\cdots\psi_{j_a}^{k_a}), \\
v & = & \prod_{i=1}^p x_i^{\beta(x_i)} \prod_{j=1}^q y_j^{\beta(y_j)} \prod_{k=1}^r z_k^{\beta(z_k)} \prod_{i=1}^p \varphi_i^{\beta(\varphi_i)}  (\psi_{j'_1}^{k'_1}\psi_{j'_2}^{k'_2}\cdots\psi_{j'_{a'}}^{k'_{a'}}),
\end{eqnarray*}
where
\[
j_1 \leq \cdots \leq j_a, \, k_1 \leq \cdots \leq k_a, \,
j'_1 \leq \cdots \leq j'_{a'}, \, k'_1 \leq \cdots \leq k'_{a'}.
\]

Now, suppose that $\pi(u) = \pi(v)$ and that none of the variables of $T$ divides both $u$ and $v$ simultaneously.  We then claim $u = v = 1$.  First, the comparison between the $t$-degree of $\pi(u)$ and that of $\pi(v)$ yields
\begin{eqnarray}
\label{aaa}
\sum_{i=1}^p \alpha(\varphi_i) + a = \sum_{i=1}^p \beta(\varphi_i) + a'.
\end{eqnarray}
Second, the comparison between the $x_\xi$-degree of $\pi(u)$ and that of $\pi(v)$ yields
\[
\alpha(x_\xi) + \sum_{i \neq \xi} \alpha(\varphi_i) + a = \beta(x_\xi) + \sum_{i \neq \xi} \beta(\varphi_i) + a', \, \, \, 1 \leq \xi \leq p.
\]
Let say, $\alpha(x_\xi) > 0$.  Then $\beta(x_\xi) = 0$ and $\alpha(\varphi_\xi) = 0$.  Thus
\begin{eqnarray*}
\sum_{i=1}^p \beta(\varphi_i) + a' & < &
\alpha(x_\xi) + \sum_{i=1}^p \alpha(\varphi_i) + a \\
& = &
\alpha(x_\xi) + \sum_{i \neq \xi} \alpha(\varphi_i) + a \\
& = & \sum_{i \neq \xi} \beta(\varphi_i) + a',
\end{eqnarray*}
which cannot happen.  Hence $\alpha(x_i) = \beta(x_i) = 0$ for $1 \leq i \leq p$.  Thus
\begin{eqnarray}
\label{ccc}
\sum_{i \neq \xi} \alpha(\varphi_i) + a & = & \sum_{i \neq \xi} \beta(\varphi_i) + a', \, \, \, 1 \leq \xi \leq p.
\end{eqnarray}
It follows from (\ref{aaa}) and (\ref{ccc}) that $\alpha(\varphi_\xi) = \beta(\varphi_\xi)$ for $1 \leq \xi \leq p$.  Hence
\[
\alpha(\varphi_i) = \beta(\varphi_i) = 0, \, \, \, 1 \leq i \leq p.
\]
In other words, one has $a = a'$ and
\begin{eqnarray*}
u & = & \prod_{j=1}^q y_j^{\alpha(y_j)} \prod_{k=1}^r z_k^{\alpha(z_k)} (\psi_{j_1}^{k_1}\psi_{j_2}^{k_2}\cdots\psi_{j_a}^{k_a}), \\
v & = & \prod_{j=1}^q y_j^{\beta(y_j)} \prod_{k=1}^r z_k^{\beta(z_k)} (\psi_{j'_1}^{k'_1}\psi_{j'_2}^{k'_2}\cdots\psi_{j'_{a}}^{k_{a}}).
\end{eqnarray*}
Let, say, $\alpha(y_\xi) > 0$ with $1 \leq \xi < q$.  Then $\beta(y_\xi) = 0$ and $\xi \not\in \{j_1, \ldots, j_a\}$.  Hence $y_\xi^{a+1}$ divides $\pi(u)$.  However, $y_\xi^{a+1}$ cannot divide $\pi(v)$, which is impossible.  Thus
\[
\alpha(y_1) = \cdots = \alpha(y_{q-1}) = \beta(y_1) = \cdots = \beta(y_{q-1}) = 0.
\]
Similarly, one has
\[
\alpha(z_2) = \cdots = \alpha(z_{r}) = \beta(z_2) = \cdots = \beta(z_{r}) = 0.
\]
As a result,
\begin{eqnarray*}
u & = & y_q^{\alpha(y_q)} z_1^{\alpha(z_1)} (\psi_{j_1}^{k_1}\psi_{j_2}^{k_2}\cdots\psi_{j_a}^{k_a}), \\
v & = & y_q^{\beta(y_q)} z_1^{\beta(z_1)} (\psi_{j'_1}^{k'_1}\psi_{j'_2}^{k'_2}\cdots\psi_{j'_{a}}^{k_{a}}).
\end{eqnarray*}
Let $\alpha(y_q) > 0$.  Then $\beta(y_q) = 0$.  Since $\deg \pi(u) = \deg \pi(v)$, one has $\alpha(z_1) = 0$ and $\beta(z_1) = \alpha(y_q)$.  Hence
\begin{eqnarray*}
u & = & y_q^{\alpha} (\psi_{j_1}^{k_1}\psi_{j_2}^{k_2}\cdots\psi_{j_a}^{k_a}), \\
v & = & z_1^{\alpha} (\psi_{j'_1}^{k'_1}\psi_{j'_2}^{k'_2}\cdots\psi_{j'_{a}}^{k_{a}}).
\end{eqnarray*}
The comparison between the $\yb$-degree of $\pi(u)$ and that of $\pi(v)$ yields $\alpha = 0$.  Finally,
\begin{eqnarray*}
u & = & \psi_{j_1}^{k_1}\psi_{j_2}^{k_2}\cdots\psi_{j_a}^{k_a}, \\
v & = & \psi_{j'_1}^{k'_1}\psi_{j'_2}^{k'_2}\cdots\psi_{j'_{a}}^{k_{a}}.
\end{eqnarray*}
Let say, $j_1 < j'_1$.  Then $y_{j_1}^a$ cannot divide $\pi(u)$.  Since $j_1 < j'_1 \leq \cdots \leq j'_a$, it follows that $y_{j_1}^a$ divides $\pi(v)$.  This is impossible.  Thus $j_1 = j'_1$.  Hence $k_1 \neq k'_1$.  Let, say, $k_1 < k'_1$.  Then $z_{k_1}^a$ cannot divide $\pi(u)$ and can divide $\pi(v)$.  Again this is impossible.  Hence $a = 0$ and $u = v = 1$, as desired.
\end{proof}

{\em Path graphs}. Let $P_n$ denote the path graph on $[n]$.  Thus the set of edges of $P_n$ is
\[
E(P_n) = \{ \{1,2\}, \{2,3\}, \ldots, \{n-1,n\}\}.
\]
Let $S = K[x_1, \ldots, x_n]$ denote the polynomial ring in $n$ variables over a field $K$ and $<_{\rm lex}$ the pure lexicographic order on $S$ induced by the ordering $x_1 > x_2 > \cdots > x_n$.  Let  $I_{P_n} \subset S$ denote the vertex cover ideal of $P_n$ with the minimal generating set of monomials $\{u_1, \ldots, u_s\}$, where
\[
u_s <_{\rm lex} \cdots <_{\rm lex} u_2 <_{\rm lex} u_1,
\]
and $\Rc(I_{P_n})$ the Rees algebra of $I_{P_n}$.  Thus
\[
\Rc(I_{P_n}) = K[x_1, \ldots, x_n, u_1 t, \ldots, u_s t] \subset S[t].
\]
Let $T = K[x_1, \ldots, x_n, y_1, \ldots, y_s]$ denote the polynomial ring in $n + s$ variables over $K$ and define the surjective ring homomorphism $\pi : T \to \Rc(I_{P_n})$ by setting
\[
\pi(x_i) = x_i, \, \pi(y_j) = u_j t, \, \, \, \, \, 1 \leq i \leq n, \, 1 \leq j \leq s.
\]
The defining ideal $J_{P_n} \subset T$ of $\Rc(I_{P_n})$ is the kernel of $\pi$.  We consider the pure lexicographic order $<_{\rm lex}$ on $T$ induced by the ordering
\[
y_1 > y_2 > \cdots > y_s > x_1 > x_2 > \cdots > x_n.
\]
Our mission is to compute the initial ideal ${\rm in}_{<_{\lex}}(J_{P_n})$ of $J_{P_n}$ with respect to $<_{\lex}$.
\begin{itemize}
\item[(i)]
Let $1 \leq i \leq n - 3$.  Suppose that $x_{i-1}x_{i}$ divides $x_0 u_j$ and that $x_{i+3}$ does not divide $u_j$.  Let $u_k = x_{i+1}u_j/x_i$.  Then $j < k$ and
\begin{eqnarray*}
\label{type_a}
x_{i+1} y_j - x_{i} y_k \in J_{P_n}.
\end{eqnarray*}
Furthermore, if $x_{n-2}x_{n-1}$ divides $u_j$ and if $u_k = x_n u_j/x_{n-1}$, then $j < k$ and $$x_{n} y_j - x_{n-1} y_k \in J_{P_n}.$$
\item[(ii)]
Let $1 \leq i \leq n - 2$.  Suppose that each of $x_{i-1}x_{i}$ and $x_{i+2}x_{i+3}$ divides $x_0 u_j x_{n+1}$.  Let $u_k = x_{i+1}u_j/(x_ix_{i+2})$.  Then $j < k$ and
\begin{eqnarray*}
\label{type_b}
x_{i+1} y_j - x_{i}x_{i+2} y_k \in J_{P_n}.
\end{eqnarray*}
\item[(iii)]
Let $j < k$.  Suppose that $x_{i-1}x_{i}$ divides $u_j$ and that neither $x_{i-2}$ nor $x_{i}$ divides $u_k$.  Let $u_j = v_j x_{i-1}x_{i} w_j$ and $u_k = v_k x_{i-1} w_k$ with $\supp(v_j) \subset [i-3]$, $\supp(w_j) \subset [n] \setminus [i+1]$, $\supp(v_k) \subset [i-3]$ and $\supp(w_k) \subset [n] \setminus [i]$.  Let $u_a = v_j x_{i-1} w_k$ and $u_b = v_k x_{i-1}x_{i} w_j$.  Furthermore, suppose that $v_j \neq v_k$.  Then $j < a$ and $j < b$ with
\begin{eqnarray*}
\label{type_c}
y_j y_k - y_a y_b \in J_{P_n}.
\end{eqnarray*}
\item[(iv)]
Let $j < k$. Suppose that $x_{i}x_{i+1}$ divides $u_j$ and that $x_{i-1}x_{i}$ divides $u_k$.  Let $u_j = v_j x_{i}x_{i+1} w_j$ and $u_k = v_k x_{i-1}x_{i} w_k$ with $\supp(v_j) \subset [i-2]$, $\supp(w_j) \subset [n] \setminus [i+2]$, $\supp(v_k) \subset [i-3]$ and $\supp(w_k) \subset [n] \setminus [i+1]$.  Let $u_a = v_j x_{i} w_k$ and $u_b = v_k x_{i-1}x_{i+1} w_j$.  Then $j < a$ and $j < b$ with
\begin{eqnarray*}
\label{type_d}
y_j y_k - x_{i} y_a y_b \in J_{P_n}.
\end{eqnarray*}
\item[(v)]
Let $j < k$. Suppose that $x_{i-1}x_{i}$ divides $u_j$ and that $x_{i}x_{i+1}$ divides $u_k$.  Let $u_j = v_j x_{i-1}x_{i} w_j$ and $u_k = v_k x_{i}x_{i+1} w_k$ with $\supp(v_j) \subset [i-3]$, $\supp(w_j) \subset [n] \setminus [i+1]$, $\supp(v_k) \subset [i-2]$ and $\supp(w_k) \subset [n] \setminus [i+2]$.  Let $u_a = v_j x_{i-1} x_{i+1} w_k$ and $u_b = v_k x_i w_j$.  Then $j < a$ and $j < b$ with
\begin{eqnarray*}
\label{type_e}
y_j y_k - x_{i} y_a y_b \in J_{P_n}.
\end{eqnarray*}
\end{itemize}

\begin{Example}
\label{n=8}
{\em
Let $n = 8$.  The vertex cover ideal of $P_8$ is generated by
\[
u_1 = x_1x_3x_4x_6x_7, \, \,
u_2 = x_1x_3x_4x_6x_8, \, \,
u_3 = x_1x_3x_5x_6x_8, \, \,
u_4 = x_1x_3x_5x_7,
\]
\[
u_5 = x_2x_3x_5x_6x_8, \, \,
u_6 = x_2x_3x_5x_7, \, \,
u_7 = x_2x_4x_5x_7, \, \,
u_8 = x_2x_4x_6x_7, \, \,
u_9 = x_2x_4x_6x_8.
\]
The binomials of types (i), (ii), (iii), (iv) and (v) are as follows:
\begin{itemize}
\item
type (i) binomials
\[
x_8y_1 - x_7y_2, \, \,
x_5y_2 - x_4y_3, \, \,
x_2y_3 - x_1y_5, \, \,
x_2y_4 - x_1y_6,
\]
\[
x_4y_6 - x_3y_7, \, \,
x_6y_7 - x_5y_8, \, \,
x_8y_8 - x_7y_9.
\]
\item
type (ii) binomials
\[
x_7y_3 - x_6x_8y_4, \, \,
x_4y_5 - x_3x_5y_9, \, \,
x_7y_5 - x_6x_8y_6.
\]
\item
type (iii) binomials
\[
y_1y_9 - y_2y_8, \, \, y_3y_6 - y_4y_5.
\]
\item
type (iv) binomials
\[
y_1y_3 - x_6y_2y_4, \, \,
y_1y_5 - x_6y_2y_6, \, \,
y_1y_5 - x_3y_3y_8, \, \,
y_1y_6 - x_3y_4y_8,
\]
\[
y_2y_5 - x_3y_3y_9, \, \,
y_2y_6 - x_3y_4y_9, \, \,
y_3y_7 - x_5y_4y_9, \, \,
y_5y_7 - x_5y_6y_9.
\]
\item
type (v) binomials
\[
y_1y_7 - x_4y_4y_8, \, \,
y_2y_7 - x_4y_4y_9, \, \,
y_3y_8 - x_6y_4y_9, \, \,
y_5y_8 - x_6y_6y_9.
\]
\end{itemize}
It then follows that the set $\Gc_{<_{\lex}}(J_{P_8})$ of the above binomials is a Gr\"obner basis of $J_{P_8}$ with respect to $<_{\lex}$.  Furthermore,
\[
\left(\Gc_{<_{\lex}}(J_{P_8}) \setminus \{y_1y_5 - x_6y_2y_6, y_1y_5 - x_3y_3y_8\}\right) \bigcup \{y_1y_5 - x_3x_6y_4y_9\}
\]
is the reduced Gr\"obner basis of $J_{P_8}$ with respect to $<_{\lex}$.
}
\end{Example}

We now approach the initial ideal of $J_{P_n}$ with respect to $<_{\lex}$.  Let $\Ac_n$ denote the set of those monomials $x_{i+1}y_j$ for which $x_{i-1}x_i$ divides $x_0 u_j$, where $1 \leq i < n$.  Let $\Bc_n$ denote the set of those monomials $y_jy_k$ with $j < k$ satisfying one of the following conditions:
\begin{itemize}
\item
there is $3 \leq i \leq n - 1$ for which $x_{i-1}x_{i}$ divides $u_j$, neither $x_{i-2}$ nor $x_{i}$ divides $u_k$, and ${\rm supp}(u_j) \cap [i-3] \neq {\rm supp}(u_k) \cap [i-3]$;
\item
there is $3 \leq i \leq n - 2$ for which $x_{i}x_{i+1}$ divides $u_j$, and $x_{i-1}x_{i}$ divides $u_k$;
\item
there is $3 \leq i \leq n - 2$ for which $x_{i-1}x_{i}$ divides $u_j$, and $x_{i}x_{i+1}$ divides $u_k$.
\end{itemize}

\begin{Theorem}
\label{path_initial}
The minimal set of monomial generators of the initial ideal ${\rm in}_{<_{\lex}}(J_{P_n})$ of $J_{P_n}$ with respect to $<_{\lex}$ is equal to $\Ac_n \cup \Bc_n$.  Thus in particular ${\rm in}_{<_{\lex}}(J_{P_n})$ is generated by quadratic monomials.
\end{Theorem}

\begin{proof}
Let $\Omega$ denote the set of those monomials $w \in T$ for which none of the monomials belonging to $\Ac_n \cup \Bc_n$ divides $w$.  By virtue of \cite[Lemma 1.1]{AHH} our task is to show that, for $w \in \Omega$ and $w' \in \Omega$ with $w \neq w'$, one has $\pi(w) \neq \pi(w')$.  Let
\[
w = x_{i_1} \cdots x_{i_p} y_{j_1} \cdots y_{j_q}, \, \, \, \, \, i_1 \leq \cdots \leq i_p, \, \, \, j_1 \leq \cdots \leq j_q,
\]
\[
w' = x_{i'_1} \cdots x_{i'_{p'}} y_{j'_1} \cdots y_{j'_q}, \, \, \, \, \, i'_1 \leq \cdots \leq i'_{p'}, \, \, \, j'_1 \leq \cdots \leq j'_q
\]
be monomials belonging to $\Omega$ with $\pi(w) = \pi(w')$.

Let $A$ be the set of those $j_a$ for which $x_1$ divides $u_{j_a}$ and $B$ the set of those $j_b$ for which $x_2$ divides $u_{j_a}$.  Let $A'$ be the set of those $j'_{a'}$ for which $x_1$ divides $u_{j'_{a'}}$ and $B'$ the set of those $j'_{b'}$ for which $x_2$ divides $u_{j'_{b'}}$.  Suppose $|A| > |A'|$.  Then $|B| < |B'|$.  Thus $x_{i_1} = x_2$.  Hence $x_{i_1}y_{j_1} \in \Ac_n$, which contradicts to $w \in \Omega$.  Thus $|A| = |A'|$ and $|B| = |B'|$.

Now, we introduce the matrix $A_w = (a_{\xi\zeta}) \in \ZZ^{q \times n}$, where $a_{\xi\zeta} = 1$ if $x_\zeta$ divides $u_{j_\xi}$ and $a_{\xi\zeta} = 0$ otherwise.  The matrix $A_{w'} = (a'_{\xi\zeta})$ is introduced similarly.  Let $\ab_1, \ldots, \ab_n$ be the column vectors of $A_w$ and $\ab'_1, \ldots, \ab'_n$ those of $A_{w'}$.  One has $\ab_1 = \ab'_1$ and $\ab_2 = \ab'_2$.

Let $3 \leq c < n$.  Suppose that $\ab_\zeta = \ab'_\zeta$ for $1 \leq \zeta \leq c - 1$.  We then claim $\ab_c = \ab'_c$.  Let $a_{1, c} = a'_{1, c}, a_{2, c} = a'_{2, c}, \ldots, a_{r-1, c} = a'_{r-1, c}$ and $a_{r, c} \neq a'_{r, c}$.  Let, say, $a_{r, c} = 1$ and $a'_{r, c} = 0$.  Then $a_{r, c - 1} = a'_{r, c - 1} = 1$.  If $r_0 > r$ and
\[
(a_{r, 1}, \ldots, a_{r, c - 1}) = (a_{r_0, 1}, \ldots, a_{r_0, c - 1}),
\]
then
\[
(a'_{r, 1}, \ldots, a'_{r, c - 1}) = (a'_{r_0, 1}, \ldots, a'_{r_0, c - 1}).
\]
Hence $a'_{r_0, c} = 0$, since $a'_{r, c} = 0$.  Thus $a'_{r_0, c+1} = 1$.  Furthermore, since $a_{r, c-1} = a_{r, c} = 1$, for $r_1 > r$ with
\[
(a_{r, 1}, \ldots, a_{r, c - 1}) \neq (a_{r_1, 1}, \ldots, a_{r_1, c - 1}),
\]
one has $a_{r_1, c} = 1$ and $a_{r_1, c+1} = 0$.  Since $a_{r, c+1} = 0$ and $a'_{r, c+1} = 1$, it follows that
\[
|\{a_{r, c+1}, a_{r+1, c+1}, \cdots, a_{q, c+1}\} \cap \{ 1 \}|
< |\{a'_{r, c+1}, a'_{r+1, c+1}, \cdots, a'_{q, c+1}\} \cap \{ 1 \}|.
\]
Since $a_{r, c-1} = a_{r, c} = 1$, it follows that, for $r_2 > r$ with $a_{r_2, c} = 1$, one has $a_{r_2, c+1} = 0$.  Thus
\[
|\{a_{1, c+1}, a_{2, c+1}, \cdots, a_{r-1, c+1}\} \cap \{ 1 \}|
\leq |\{a'_{1, c+1}, a'_{2, c+1}, \cdots, a'_{r-1, c+1}\} \cap \{ 1 \}|.
\]
Hence
\[
|\{a_{1, c+1}, a_{2, c+1}, \cdots, a_{q, c+1}\} \cap \{ 1 \}|
< |\{a'_{1, c+1}, a'_{2, c+1}, \cdots, a'_{q, c+1}\} \cap \{ 1 \}|.
\]
Thus $c+1 \in \{i_1, \ldots, i_p\}$.  Since $a_{r, c-1} = a_{r,c} = 1$, one has $x_{c+1}y_{j_r} \in \Ac_n$, which contradicts $w \in \Omega$.  Hence $\ab_c = \ab'_c$, as desired.

The discussion done above yields $\ab_\zeta = \ab'_\zeta$ for $1 \leq \zeta < n$.  Clearly, if $\ab_{n-1} = \ab'_{n-1}$, then $\ab_{n} = \ab'_n$.  Thus $A_w = A_{w'}$.  In other words, $y_{j_1} = y_{j'_1}, \ldots, y_{j_q} = y_{j'_q}$.  Since $\pi(w) = \pi(w')$, it follows that $p = p'$ and $x_{i_1} = x_{i'_1}, \ldots, x_{i_p} = x_{i'_p}$.  Hence $w = w'$, as required.
\end{proof}

\begin{Corollary}
\label{path_GB}
The set of binomials of types (i), (ii), (iii), (iv) and (v) is a Gr\"obner basis of $J_{P_n}$ with respect to $<_{\lex}$.
\end{Corollary}

{\em Cameron-Walker graphs whose bipartite part is complete.} A finite simple graph $G$ is called {\em Cameron--Walker} \cite{HHKO} if $G$ consists of a connected bipartite graph with vertex partition $[n] \cup [m]$, where $n \geq 1$ and $m \geq 1$, such that there is at least one leaf edge attached to each vertex $i \in [n]$ and that there may be possibly some pendant triangles attached to each vertex $j \in [m]$.  For example the graph depicted in Figure~\ref{Figure2},

%%%%%%%%%%%%%%%%%%%%%%%%%%%%%%
\begin{figure}[h]
\begin{center}
\includegraphics[height=4cm]{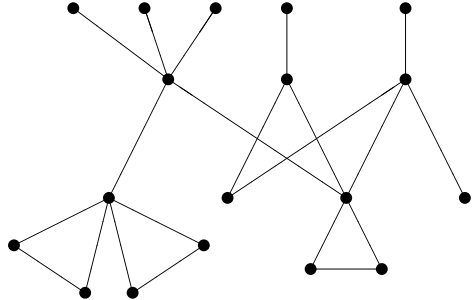}
\end{center}
\caption{A Cameron-Walker graph $G$}
\label{Figure2}
\end{figure}
%%%%%%%%%%%%%%%%%%%%%%%%%%%%%%

\noindent
is a Cameron--Walker graph with $n = 3$ and $m = 4$.  Every Cameron--Walker graph is sequentially Cohen--Macaulay \cite[Theorem 3.1]{HHKO} and hence the vertex cover ideal of a Cameron--Walker graph is componentwise linear.  Furthermore, a Cameron--Walker graph $G$ is Cohen--Macaulay if and only if $G$ consists of a connected bipartite graph with vertex partition $[n] \cup [m]$ such that there is exactly one leaf edge attached to each vertex $i \in [n]$ and that there is exactly one pendant triangle attached to each vertex $j \in [m]$.

It is natural to ask whether the vertex cover ideal of a Cameron--Walker graph can satisfy the $x$-condition.  However, to study the question for arbitrary Cameron--Walker graph seems to be rather difficult.  Now, we focus attention on Cameron--Walker graphs whose bipartite part are complete bipartite graphs.

Let $G$ be a Cameron--Walker graph whose bipartite part is the complete bipartite graph $K_{n,m}$ with vertex partition
$$\{\xi_1, \ldots, \xi_n\} \cup \{\zeta_1, \ldots, \zeta_m\},$$
where $n \geq 1$ and $m \geq 1$, such that each vertex $\xi_i$ possesses leaves $$\{\xi_i, a_i^{(1)}\}, \ldots, \{\xi_i, a_i^{(p_i)}\}$$ with $p_i \geq 1$ and that each vertex $\zeta_j$ possesses pendant triangles $$\{\zeta_j, b_j^{(1)}, c_j^{(1)}\}, \ldots, \{\zeta_j, b_j^{(q_j)}, c_j^{(q_j)}\}$$ with $q_j \geq 0$.

Let $S$ denote the polynomial ring over a field $K$ whose variables coincide with the vertices of $G$ and $<_{\rm lex}$ the pure lexicographic order on $S$ induced by the ordering
\[
a_1^{(1)} > a_1^{(2)} > \cdots > a_1^{(p_1)} > a_2^{(1)} > \cdots > a_2^{(p_2)} > \cdots > a_n^{(1)} > \cdots > a_n^{(p_n)}
\]
\[
> b_1^{(1)} > c_1^{(1)} > \cdots > b_1^{(q_1)} > c_1^{(q_1)} > \cdots > b_m^{(1)} > c_m^{(1)} > \cdots > b_m^{(q_m)} > c_m^{(q_m)}
\]
\[
> \zeta_1 > \cdots > \zeta_m > \xi_1 > \cdots > \xi_n.
\]

Let $I_G \subset S$ denote the vertex cover ideal of $G$ whose minimal set of monomial generators is $\{u_1, \ldots, u_s\}$, where
\[
u_s <_{\rm lex} \cdots <_{\rm lex} u_2 <_{\rm lex} u_1.
\]
In particular,
\[
 u_1 = a_1^{(1)} \cdots a_n^{(p_n)} \zeta_1 \cdots \zeta_m
b_1^{(1)} \cdots b_m^{(q_m)},
\]
\[
u_s = \xi_1 \cdots \xi_n \zeta_1 \cdots \zeta_m c_1^{(1)} \cdots c_m^{(q_m)}.
\]

Let $\Rc(I_G)$ be the Rees algebra of $I_G$, which is the subring of $S[t]$ generated by
\[
a_1^{(1)}, \ldots, a_n^{(p_n)}, b_1^{(1)}, \ldots, b_m^{(q_m)}, c_1^{(1)}, \ldots, c_m^{(q_m)}, \xi_1, \ldots, \xi_n, \zeta_1, \ldots, \zeta_m, u_1 t, \ldots, u_s t.
\]

Let $T$ denote the polynomial ring over $K$ whose variables are the vertices of $G$ together with $y_1, \ldots, y_s$. One defines the surjective ring homomorphism $\pi: T \to \Rc(I_G)$ by setting $\pi(x) = x$ if $x$ is a vertex of $G$ and $\pi(y_j) = u_j t$ for $1 \leq j \leq s$.  The defining ideal $J_{G} \subset T$ of $\Rc(I_G)$ is the kernel of $\pi$.  Let $<_{\rm rev}$ denote the reverse lexicographic order on $K[y_1, \ldots, y_s]$ induced by the ordering
\[
y_1 > y_2 > \cdots > y_s.
\]
% For example, one has $y_1^3y_4 <_{\rm rev} y_2^2y_3^2 <_{\rm rev} y_1^3y_4^2$.

Now, one introduces the monomial order $<$ on $T$ defined as follow: If $v = v'v''$ and $w = w'w''$ are monomials belonging to $T$, where $v', w'\in S$ and where $v'', w'' \in K[y_1, \ldots, y_s]$, then $v < w$ if either (i) $v'' <_{\rm rev} w''$ or (ii) $v'' = w''$ and $v' <_{\rm lex} w'$.  For example, one has
\[
(a_2^{(1)})^5 y_1^3y_4 < (a_1^{(1)})^3 y_1^3y_4 < \xi_n\zeta_my_2^2y_3^2 < y_1^3y_4^2.
\]

Our job is to compute the initial ideal ${\rm in}_{<}(J_G)$ of $J_G$ with respect to $<$.
\begin{itemize}
\item[(i)]
If $\xi_i$ does not divide $u_{r}$ and if $u_{r'} = \xi_i u_{r}/(\prod_{k=1}^{p_i} a_i^{(k)}) w_0$, where $w_0 = \prod_{q_j = 0} \zeta_j$ if each $\xi_{i'}$ with $i' \neq i$ divides $u_{r}$, and otherwise $w_0 = 1$.  Then $r < r'$ and
\[
\xi_i y_{r} - \prod_{k=1}^{p_i} a_i^{(k)} w_0 y_{r'} \in J_G
\]
with $\xi_i y_{r}$ its initial monomial.
\item[(ii)]
If $q_j > 0$ and if $\zeta_j$ does not divide $u_r$ and if $u_{r'} = \zeta_j u_r/\prod_{k=1}^{q_j} b_j^{(k)}$, then $r < r'$ and
\[
\zeta_j y_r - \prod_{k=1}^{q_j} b_j^{(k)} y_{r'} \in J_G
\]
with $\zeta_j y_r$ its initial monomial.
\item[(iii)]
If $c_j^{(k)}$ does not divide $u_r$ and if $u_{r'} = c_j^{(k)} u_r/b_j^{(k)}$, then $r < r'$ and
\[
c_j^{(k)} y_r - b_j^{(k)} y_{r'} \in J_G
\]
with $c_j^{(k)} y_r$ its initial monomial.
\item[(iv)]
Fix $i$ and $r_0 \neq r_1$ with $r_0 > r_1$.  Suppose that $\xi_{i}$ does not divide $u_{r_0}$ and that $\xi_{i}$ divides $u_{r_1}$.  Let $u_{r'_0} =  \xi_{i}u_{r_0}/(\prod_{k=1}^{p_{i}} a_{i}^{(k)})w_0$, where $w_0 = \prod_{q_j = 0} \zeta_j$ if each $\xi_{i'}$ with $i' \neq i$ divides $u_{r_0}$, and otherwise $w_0 = 1$.  Let $u_{r'_1} = (\prod_{k=1}^{p_{i}} a_{i}^{(k)})u_{r_1}/\xi_{i}$.  Then $r'_0 > r_0 > r_1 > r'_1$ and
\[
y_{r_0}y_{r_1} - w_0 y_{r'_0}y_{r'_1} \in J_G
\]
with $y_{r_0}y_{r_1}$ its initial monomial.
\item[(v)]
Fix $j$ and $r_0 \neq r_1$ with $r_0 > r_1$.  Suppose that $\zeta_j$ divides both $u_{r_0}$ and $u_{r_1}$.  Fix $k$ and suppose that $b_j^{(k)}$ divides $u_{r_0}$ and that $c_j^{(k)}$ divides $u_{r_1}$.
Let $u_{r'_0} = c_j^{(k)}u_{r_0}/b_j^{(k)}$ and $u_{r'_1} = b_j^{(k)}u_{r_1}/c_j^{(k)}$.  Then $r'_0 > r_0 > r_1 > r'_1$ and
\[
y_{r_0}y_{r_1} - y_{r'_0}y_{r'_1} \in J_G
\]
with $y_{r_0}y_{r_1}$ its initial monomial.
\item[(vi)]
Let, in general, for each $1 \leq j \leq m$ and for each $1 \leq r \leq s$, $\Bc_r(j)$ denote the set of those $b_j^{(k)}$ which divide $u_r$ and $\Cc_r(j)$ that of those $c_j^{(k)}$ which divides $u_r$.  Let $w_r(j) = (\prod_{b_j^{(k)}\in \Bc_r(j)}b_j^{(k)})(\prod_{c_j^{(k)}\in \Cc_r(j)}c_j^{(k)})$.  Fix $j$ and $r_0 \neq r_1$ with $r_0 > r_1$ for which each $\xi_i$ divides both $u_{r_0}$ and $u_{r_1}$.  Suppose that $\zeta_{j}$ does not divide $u_{r_0}$ and that $\zeta_{j}$ divides $u_{r_1}$.  Let $u_{r'_0} = \zeta_j w_{r_1}(j) u_{r_0}/w_{r_0}(j)$ and $u_{r'_1} = w_{r_0}(j)u_{r_1}/\zeta_j w_{r_1}(j)$.  Then $r'_0 > r_0 > r_1 > r'_1$ and
\[
y_{r_0}y_{r_1} - y_{r'_0}y_{r'_1} \in J_G
\]
with $y_{r_0}y_{r_1}$ its initial monomial.
\end{itemize}

\begin{Example}
\label{example}
{\em
Let $G$ denote the Cameron--Walker graph with
\[
n = 4, \, m = 3, \, p_1 = 3, \, p_2 = 1, \, p_3 = 2, \, p_4 = 1, \, q_1 = 2, \, q_2 = 0, \, q_3 = 1
\]
whose bipartite part is the complete graph $K_{4,3}$.  If
\[
u_{r_1} = \xi_1 \xi_4 \zeta_1 \zeta_2 \zeta_3 a_2^{(1)} a_3^{(1)} a_3^{(2)} b_1^{(1)} c_1^{(2)} c_3^{(1)},
\]
then $\xi_2 y_{r_1}$ is an initial monomial of type (i) with $w_0 = 1$ and $c_1^{(1)} y_{r_1}$ is that of type (iii).  If
\[
u_{r_2} = \xi_1 \xi_2 \xi_3 \xi_4 \zeta_3 b_1^{(1)} c_1^{(1)} b_1^{(2)} c_1^{(2)} c_3^{(1)},
\]
then $\zeta_1 y_{r_2}$ is an initial monomial of type (ii).  If
\[
u_{r_3} = \xi_1 \xi_2 \xi_4 \zeta_1 \zeta_2 \zeta_3 a_3^{(1)} a_3^{(2)} b_1^{(1)} c_1^{(2)} b_3^{(1)},
\]
\[
u_{r_4} = \xi_1 \xi_3 \xi_4 \zeta_1 \zeta_2 \zeta_3 a_2^{(1)} c_1^{(1)} b_1^{(2)} c_3^{(1)},
\]
then $r_3 > r_4$ and $y_{r_3}y_{r_4}$ is an initial monomial of type (iv) with $w_0 = \zeta_2$ as well as that of type (v).  If
\[
u_{r_5} = \xi_1 \xi_2 \xi_3 \xi_4 \zeta_1 b_1^{(1)} b_1^{(2)} b_3^{(1)} c_3^{(1)},
\]
then $r_2 < r_5$ and $y_{r_2}y_{r_5}$ is an initial monomial of type (vi).
}
\end{Example}

We now come to the initial ideal ${\rm in}_{<}(J_G)$ of $J_G$ with respect to $<$.  Let $\Wc_G$ denote the set of initial monomials of binomials arising from (i), (ii), (iii), (iv), (v) and (vi) as described above.

\begin{Theorem}
\label{CWinitial}
The minimal set of monomial generators of the initial ideal ${\rm in}_{<}(J_G)$ of $J_G$ with respect to $<$ is equal to $\Wc_G$.  Thus in particular ${\rm in}_{<}(J_G)$ is generated by quadratic monomials.
\end{Theorem}

\begin{proof}
Let $\Omega$ denote the set of those monomials $w \in T$ for which none of the monomials belonging to $\Wc_G$ divides $w$. By virtue of \cite[Lemma 1.1]{AHH} our task is to show that, for $w \in \Omega$ and $w' \in \Omega$ with $w \neq w'$, one has $\pi(w) \neq \pi (w')$.  Let
\[
w = g \cdot y_{\delta_1} \cdots y_{\delta_e}, \, \, \, \, \, \, \, \, \, \, g \in S, \, \, \, \delta_1 \leq \cdots \leq \delta_e,
\]
\[
w' = g' \cdot y_{\delta'_1} \cdots y_{\delta'_e}, \, \, \, \, \, \, \, \, \, \, g' \in S, \, \, \, \delta'_1 \leq \cdots \leq \delta'_e
\]
be monomials belonging to $\Omega$ with $\pi(w) = \pi(w')$.

Let $\Ac(r)$ denote the set of those variables $\xi_i$ which divide $u_r$.  One claims
\[
\Ac(\delta_1) \subset \cdots \subset \Ac(\delta_e).
\]
In fact, if $\xi_i \in \Ac(\delta_\rho) \setminus \Ac(\delta_{\rho+1})$, then $w$ is divided by the initial monomial of a binomial of type (iv), since $\delta_\rho < \delta_{\rho+1}$, which contradicts $w \in \Omega$.  One has also
\[
\Ac(\delta'_1) \subset \cdots \subset \Ac(\delta'_e).
\]
It then follows that
\[
\Ac(\delta_1) = \Ac(\delta'_1), \ldots, \Ac(\delta_e) = \Ac(\delta'_e).
\]
To see why this is true, suppose that
\[
\Ac(\delta_1) = \Ac(\delta'_1), \ldots, \Ac(\delta_{\rho-1}) = \Ac(\delta'_{\rho-1}), \Ac(\delta_{\rho}) \neq \Ac(\delta'_{\rho}),
\]
and say, $\xi_i \in \Ac(\delta_\rho) \setminus \Ac(\delta'_{\rho})$.  Since $\pi(w) = \pi(w')$, the monomial $g'$ is divided by $\xi_i$.  Then $w$ is divided by the initial monomial of a binomial of type (i), since $\xi_i$ does not divide $u_{\delta'_{\rho}}$, which contradicts $w' \in \Omega$.  One can choose $\rho_0$ with
\[
\Ac(\delta_1) \subset \cdots \subset \Ac(\delta_{\rho_0}) \neq \Ac(\delta_{\rho_0 + 1}) = \cdots = \Ac(\delta_e) = \{\xi_1, \ldots, \xi_n\}.
\]
Then each $\zeta_j$ divides both $u_{\delta_\rho}$ and $u_{\delta'_\rho}$ for $\rho \leq \rho_0$.

Let $1 \leq j \leq m$ with $q_j > 0$.  Let $\rho_0 < \rho' < \rho'' \leq e$.  If $\zeta_j$ divides $u_{\rho'}$, then by using the initial monomials of binomials of type (vi), it follows that $\zeta_j$ divides $u_{\rho''}$.  Suppose that $\zeta_j$ divides each of $u_{\delta_{\rho'}}$ with $\rho_1 \leq \rho' \leq e$ and does not divide $u_{\delta_{\rho_1-1}}$.  If $\zeta_j$ does not divide $u_{\delta'_{\rho_1}}$, then $\zeta_j$ divides $g'$.  Thus $w'$ is divided by the initial monomial of a binomial of type (ii), since $\zeta_j$ does not divides $u_{\delta'_{\rho_1}}$, which contradicts to $w' \in \Omega$.  Hence $\zeta_j$ divides each of $u_{\delta'_{\rho'}}$ with $\rho_1 \leq \rho' \leq e$ and in addition, does not divide $u_{\delta'_{\rho_1-1}}$.

Let $1 \leq k \leq q_j$.  Let $\rho^*$ and $\rho^\#$ belong to $\{1, \ldots, \rho_0, \rho_1, \ldots, e\}$ with $\rho^* < \rho^\#$.  If $b_j^{(k)}$ divides $u_{\delta_{\rho^\#}}$, then by using the initial monomials of binomials of type (v), it follows that $b_j^{(k)}$ divides $u_{\delta_{\rho^*}}$.  Suppose that $b_j^{(k)}$ divides $u_{\delta_{\rho^*}}$ with $\rho^* \in \{1, \ldots, \rho_2\} \cap \{1, \ldots, \rho_0, \rho_1, \ldots, e\}$ and that $c_j^{(k)}$ divides  $u_{\delta_{\rho^\#}}$ with
$$\rho^\# \in \{\rho_2+1, \ldots, e\} \cap \{1, \ldots, \rho_0, \rho_1, \ldots, e\}.$$  Furthermore, suppose that $b_j^{(k)}$ divides $u_{\delta'_{\rho^*}}$ with $\rho^* \in \{1, \ldots, \rho'_2\} \cap \{1, \ldots, \rho_0, \rho_1, \ldots, e\}$ and that $c_j^{(k)}$ divides $u_{\delta'_{\rho^\#}}$ with $\rho^\# \in \{\rho'_2+1, \ldots, e\} \cap \{1, \ldots, \rho_0, \rho_1, \ldots, e\}$.  Provided that $\rho_2$ and $\rho'_2$ belong to $\{1, \ldots, \rho_0, \rho_1, \ldots, e\}$, one has $\rho_2 = \rho'_2$.  In fact, if, say, $\rho_2 < \rho'_2$, then $c_j^{(k)}$ divides $g'$.  Hence $w'$ is divided by the initial monomial of a binomial of type (iii), since $c_j^{(k)}$ does not divide $u_{\delta'_{\rho'_2}}$, which contradicts to $w' \in \Omega$.

As a result, one finally has
\[
u_{\delta_{1}} = u_{\delta'_{1}}, \ldots, u_{\delta_{e}} = u_{\delta'_{e}}.
\]
Hence $\delta_1 = \delta'_1, \ldots, \delta_e = \delta'_e$ and $w = w'$, as desired.
\end{proof}

By Theorems ~\ref{GB}, \ref{path_initial}, ~\ref{CWinitial}, and Theorem \ref{monomialcase} we get the following

\begin{Corollary}
Let $G$ be one of the following graphs.
\begin{itemize}
  \item A path graph,
  \item A biclique graph,
  \item A Cameron-Walker graph whose bipartite graph is a complete bipartite graph.
\end{itemize}

Then all powers of the vertex cover ideal $I_G$ have linear quotients and then are componentwise linear.

\end{Corollary}

{}

\end{document}